\documentclass[10pt]{amsart}

\usepackage{geometry,graphicx,amssymb,amsmath,amsbsy,eucal,amsfonts,mathrsfs,amscd,bm,tcolorbox}
\usepackage{tkz-euclide}

\numberwithin{equation}{section}

\allowdisplaybreaks[3]

\newtheorem{theorem}{Theorem}[section]
\newtheorem{lemma}[theorem]{Lemma}

\newtheorem{corollary}[theorem]{Corollary}
\newtheorem{proposition}[theorem]{Proposition}
\theoremstyle{definition}

\theoremstyle{remark}
\newtheorem{remark}[theorem]{Remark}

\newcommand{\bF}{\bld{F}}
\newcommand{\bxi}{\bld{\xi}}

\newcommand{\jump}[1]{\left[\hspace{-0.025in}\left[#1\right]\hspace{-0.025in}\right]}

\newcommand{\curl}{{\ensuremath\mathop{\mathrm{curl}\,}}}

\newcommand{\tnorm}[1]{\vert\hspace{-0.3mm}\Vert#1\Vert\hspace{-0.3mm}\vert}

\newcommand{\bld}[1]{\boldsymbol{#1}}

\newcommand{\bY}{\bld{Y}}
\newcommand{\bv}{\bld{v}}
\newcommand{\bw}{\bld{w}}

\newcommand{\bn}{\bld{n}}
\newcommand{\bu}{\bld{u}}
\newcommand{\bhu}{\hat{\bld{u}}}
\newcommand{\bff}{\bld{f}}

\newcommand{\bU}{\bld{U}}
\newcommand{\bE}{\bld{E}}

\newcommand{\bV}{\bld{V}}

\newcommand{\bldeta}{\bld{\eta}}

\newcommand{\bPi}{\bld{\Pi}}
\newcommand{\bpi}{\bld{\pi}}

\newcommand{\bzeta}{\bld{\zeta}}

\newcommand{\bH}{\bld{H}}

\newcommand{\by}{\bld{y}}

\newcommand{\bx}{\bld{x}}

\newcommand{\bL}{\bld{L}}

\newcommand{\bbeta}{{\bm \beta}}

\newcommand{\bPsi}{{\bm \Psi}}
\newcommand{\bvarphi}{\bm \varphi}

\newcommand{\bc}{\bm c}

\title[Error estimates for the Smagorinsky turbulence model]{Error estimates for the Smagorinsky turbulence model: enhanced stability through
  scale separation and numerical stabilization}

\author{Erik Burman, Peter Hansbo, Mats G. Larson}

\begin{document}

\begin{abstract}
In the present work we show some results on
the effect of the Smagorinsky model on the stability of the associated
perturbation equation. We show that in the presence of a spectral gap, such that the
flow can be decomposed in a large scale with moderate gradient
and a small amplitude fine scale with arbitratry gradient, the
Smagorinsky model admits stability estimates for perturbations, with exponential growth
depending only on the large scale gradient. We then show in the
context of stabilized finite element methods that the same result
carries over to the approximation and that in this context, for
suitably chosen finite element spaces the Smagorinsky model acts as a
stabilizer
yielding close to optimal error estimates in the $L^2$-norm for smooth
flows in the pre-asymptotic high
Reynolds number regime. 
\end{abstract}

\maketitle

\section{Introduction}
The modelling and accurate approximation of turbulent flows at high
Reynolds numbers remain an outstanding challenge.
One of the earliest attempts at the design of a turbulence model is the
one attributed to
Smagorinsky \cite{Smag, Smag93}, where a nonlinear viscosity replaces the
Reynolds stresses, thereby providing closure to the filtered
Navier-Stokes' equations. The model is believed to be inspired by early work on shock capturing
methods for conservation laws due to Von Neumann and Richtmyer \cite{VNeu50}. 

An appealing feature is that it has been shown that the Navier-Stokes' equations
regularized system admits a unique solution \cite{Lady68}, with
enhanced regularity \cite{BdV09a, BdV09b}. It has also
been shown using scaling arguments that away from boundary layers the
dissipation rate matches the Kolmogorov decay rate \cite{La16} for
uniformly turbulent flows. On
the other hand the Smagorinsky model is considered to be too
dissipative, in particular in the laminar zone. The parameter of the
model has to
be tuned differently to obtained the best results depending the flow
problem \cite{Ge09} and on the numerical method used \cite{MGS07},
pointing to a subtle interplay between the flow and the numerical
method for the performance of the Smagorinsky model in LES. Nevertheless thanks to its simplicity and its
physical consistency, the Smagorinsky model
remains an important tool for the modelling of turbulent high Reynolds
flows. 

The objective of the present work is to show some results on
the effect of the Smagorinsky model on the stability of the fluid
dynamics and on numerical approximation. We show that in the presence of a spectral gap, such that the
flow can be decomposed in a large scale with moderate gradient
and a small amplitude fine scale with arbitrary gradient, the
Smagorinsky model admits stability estimates with exponential growth
depending only on the large scale gradient. We then show in the
context of stabilized finite element methods that the same result
carries over to the approximation and that in this context, for
suitably chosen finite element spaces the Smagorinsky model acts as a
stabilizer
yielding optimal error estimates for smooth flows in the high
Reynolds number regime. This gives some evidence that the Smagorinsky
model is efficient in situations where the flow has a clearly defined
spectral gap. This is not in general the case for turbulent flows, but
is believed to occur at high Reynolds number, near rough surfaces or
the atmospheric boundary layer \cite{GT09}, or more generally in atmospheric and geophysical flows \cite{FP70,We80,PFC83}.
Interestingly, this is the context for which the Smagoringsky model was
originally proposed \cite{Smag}. 

The notion of Large Eddy Simulation (LES) is not very well
defined. Depending on the context it is considered either as a modelling
issue, or a numerical technique. In the first case the partial
differential equation is perturbed by adding terms that model the
so-called Reynolds stress term expressing the 
effect of unresolved scales, where the smallest length scale is a model
parameter. The regularized model should be well-posed and can then be
discretized. The challenge here is on the modelling side: how can one
design a model that allows for a more stable representation of
quantities of interests? Observe that to distinguish this method from
the situation where the turbulence model is used for numerical
stability, the model parameters should be
fixed before discretization and not coupled to the mesh-size. In case
such a coupling is introduced it would need to be justified, in
particular in the asymptotic limit. A typical requirement is that the
physically relevant solution, or a so-called ``suitable'' solution is
obtained in the limit \cite{GP05}.
For so called numerical LES, or Implicit LES, on the other hand \cite{Bo07}, the Navier-Stokes' equations
are considered, but a numerically stable discretization method is used
ensuring that sufficient dissipation is added for the numerical scheme to
remain stable, but not so much that smooth structures of the flow are destroyed. The underpinning idea is
that any quantity of interest that has sufficient stability properties
will then be computable \cite{HJ07, Bu15}, if the numerical method
resolves a sufficient range of scales so that the effect of unresolved
scales is negligible. To be able to claim that an analysis of a LES
method is successful it is necessary to address the question of
stability, but there appears to be no stability estimates with
moderate exponential growth, for any
quantity of interest for the
Navier-Stokes equations in the high Reynolds number regime. The best rigorous result of quantitative
bound of perturbation growth in the inviscid
regime appears to be \cite{AD04}, where a two-dimensional flow is analyzed
and the stability constant grows with a double exponential in time. 

Finite element analysis typically relies on the
velocities having bounded gradients \cite{JS86, HS90}, with a slight
improvement possible in the two-dimensional case
using parabolic regularization of the viscous scales
\cite{Bu15}. Typically the stability constant is on the form
$\exp(K T)$, with $K$ proportional to the maximum gradient of the
velocity in the flow.

For the discussion below we will for simplicity consider the $p$-Laplacian model that acts
on the whole gradient instead of the symmetric gradient as in the
Smagorinsky case. This regularization was introduced
in \cite[Chap. 2, sec. 5]{Lions69}, where also existence and
uniqueness of solutions were studied. The below arguments can however easily be extended to the
more common versions of the Smagorinsky model using the deformation
tensor. The only essential difference is the need to apply Korn's
inequality in $L^p$, see for instance \cite[Chapter 7]{DD12} to get
control of the gradient of the solution in $L^p$ through the control
of its symmetric part.

\paragraph{{\bf {Main results and outline}}}

There appears to be very little known about the stability properties
of the Smagorinsky model, or for that matter its qualities as a
stabilizing term for numerical methods for high Reynolds flow in the
laminar regime. 
In this paper we will consider the Smagorinsky model first, as a
turbulence model and then, in a finite element method, as a
stabilizing term. In the former case we prove a stability result
(Theorem \ref{thm:red_exp})
showing that the perturbation growth is independent of high frequency,
low amplitude oscillations in the solution and that as the
perturbation error becomes larger, its growth is moderated by the
nonlinear feedback of the turbulence model.

On the other hand the use of the
Smagorinsky model as a stabilizer in the framework of divergence free
affine finite elements, allows us to prove the classical
$L^2$ error estimates of order $O(h^{\frac32})$ (where $h$ denotes the mesh-size), that are the best known estimates for stabilized finite
element methods using affine approximation for the Navier-Stokes' equations in the high Reynolds
laminar regime \cite{JS86, HS90, JRB95, BF07}. Indeed if $\bu$
denotes the solution to the Navier-Stokes equations and $\hat \bu_h$
denotes the finite element solution stabilized using the Smagorinsky
model with characteristic lengthscale $O(h)$ and the flow is laminar
in the pre asymptotic regime where the mesh-size is
larger than the viscosity,
then there (Theorem \ref{thm:error_bound}) 
\[
\|(\bu - \hat \bu_h)(T)\|_{L^2(\Omega)} \leq C(u) h^{\frac32}, \quad T>0
\]
where $C(u)$ is a constant that depends on Sobolev norms of the exact
solution to the Navier-Stokes' equations and time. In particular there
are no inverse powers of the viscosity in $C(u)$. The exponential growth of
the constant inherits the scale separation properties of the
continuous equations.

Finally we show the performance of the Smagorinsky finite element method qualitatively of
some academic test cases,
illustrating that the method returns a solution also in underresloved
computations where the standard Galerkin breaks down, but also that 
this solution is sensitive to the choice of the parameter in the
Smagorinsky model. A large parameter gives an overly diffusive
approximation on coarse meshes.

The main conclusion of this work is that the Smagorinsky model acts both on the level of
stability of perturbations and as a numerical stabilizer. This shows
that the nonlinear viscosity can be interpreted both as a turbulence
model for LES and as an implicit LES model based on stable numerical
simulation, provided the numerical method is chosen carefully to
balance and complement the dissipation properties.

\section{Smagorinsky model problem}

Let $\Omega \subset \mathbb{R}^d$, $d=2,3$, denote an
open domain with smooth (or convex polyhedral) boundary $\partial \Omega$. Consider the time interval $I = [0,T]$ and
denote the space time domain $\Omega \times I$ by $Q$.
We consider the Navier-Stokes-Smagorinsky equations on the form
\begin{equation}\label{eq:NS_smag}
\begin{array}{rl}
\partial_t \hat \bu + (\hat \bu \cdot \nabla )\hat \bu + \nabla p -\mu
\Delta \hat \bu -
\nabla \cdot \hat \nu(\hat \bu) \nabla \hat \bu& = {\bf f} \quad \mbox{ in } Q\\[3mm]
\nabla \cdot \hat \bu & = 0 \quad \mbox{ in } Q.
\end{array}
\end{equation}
The artificial viscosity matrix is diagonal with  $(\hat \nu(\bu))_{ii} =
\nu(\delta) |\nabla \bu|_F$, $i=1,\hdots,d$, where $|\cdot|_F$ denotes the Frobenius
norm, i.e. $|\nabla \bu|_F = (\sum_{i=1}^d |\nabla u_i|^2)^{1/2}$, and $\nu:\mathbb{R}^+
\mapsto \mathbb{R}^+$ a monotonically increasing function with
$\nu(0)=0$ and $\delta$ a characteristic
lengthscale. Below we will consider the choice $\nu(\delta) :=
\delta^2$ which corresponds to the well-known
Smagorinsky model \cite{Smag93}, using the gradient instead of the
deformation tensor. Taking
$\delta=0$ on the other hand results in the standard Navier-Stokes'
equations. The Smagorinsky model is not dependent on a particular
filter and the length scale $\delta$ can be allowed to vary in space
and time. As a consequence Van Driest wall damping \cite{VD56} or Germano dynamic
models can be included in the arguments below. For simplicity we here consider constant $\delta$ and
either homogeneous Dirichlet boundary conditions or periodic boundary
conditions for \eqref{eq:NS_smag}. We introduce the vectorial Bochner
space
\[
\bV:= [L^3(I;W^{1,3}(\Omega)) \cap L^\infty(I;L^2(\Omega))]^d.
\]
We will assume that the solution to the Navier-Stokes' equations,
$\bu$, has the additional regularity $\bu \in
\bV$. 
 Since \eqref{eq:NS_smag} results in an
$O(\delta^2)$ perturbation of the Navier-Stokes' equations, it is
reasonable to assume that the convergence from the
Smagorinsky-Navier-Stokes' system to the Navier-Stokes' system is
at best $O(\delta^2)$ for a smooth enough solution (to the
Navier-Stokes' system). If the solution is
not smooth enough the effect of the perturbation will be larger. For
instance, as shown below, if $\bu \in \bV$ the
perturbation due to the model is $O(\delta)$. However, the error due to
the inconsistency is not the whole story, indeed the error can
typically be written (see Theorem \ref{thm:red_exp} below)
\begin{equation}\label{eq:bound}
\|(\bu - \tilde \bu)(T)\|_{L^2(\Omega)} \leq C(\bu) \exp(\alpha(T)) \delta^\beta,
\end{equation}
where $\bu$ is the solution to the Navier-Stokes' equations, $\tilde
\bu$ is the solution to the Navier-Stokes'-Smagorinsky equations,
$C(\bu)$ is a constant only dependent on Sobolev norms of $\bu$ and
$\alpha(T)$ is a function, depending on the exact solution, quantifying the stability of the nonlinear
problem and $\beta$ is the power of the nonconsistency, typically in
the interval $(0,2)$. 
 
The consequence of this for practical computation is that for the situation where $\delta$ coincides with
the cell-width of the computational mesh, improving the approximation
order beyond that of affine finite elements (formally giving an $O(h^2)$
$L^2$-error bound) will not result in an improved approximation for
laminar flows: the consistency error $O(\delta^2)$ of the Smagorinsky model will dominate.

For laminar flows where the flow field has small
spatial variation also the coefficient $\alpha(T)$, which typically scales as
$|\bu|/\nu^{\frac12}$ or $\int_I \|\nabla \bu(\cdot,t)\|_{L^\infty(\Omega)}$, is
moderate and one can conclude that the perturbation error is
essentially determined by the size of $\delta$ in the right hand side. Decreasing $\delta$ will lead to a
decrease in the error and the natural (and trivial) choice is $\delta=0$ for which
the error is zero, since the two solutions coincide: there is no need
for a model in laminar flows.

In the less regular case, $\bu \in \bV$, which
seems to be the minimum requirement for a bound of the type
\eqref{eq:bound} to make sense, the situation is different. Now
$\alpha$ can not be assumed small, so the right hand side depends both
on the coefficient $\delta$ and the exponential growth. Except for
very short times the
exponential growth will make the bound meaningless. 


We will
show in the present paper that we can refine our definition of the
coefficient $\alpha$ by letting it  
implicitly depend on $\delta$, so that an increase in $\delta$, although
it increases the consistency error, can lead to a decrease in the
exponential growth through the nonlinear feedback of the Smagorinsky
operator. This leads to a potential moderation of the error growth through an
increase in $\delta$, which can explain
why this behavior has been observed in computations \cite{Ge09}. Observe that
this result depends on the structure of the actual solution $\bu$, so
it is not a general statement and only of qualitative nature. The
analysis is also done in the $L^2$-norm which is too strong to be a
reasonable target for turbulent flows. Unfortunately the analysis of other target
quantities appears to be very difficult due to the complexity of the
linearized adjoint to the Smagorinsky perturbation equation.

Nevertheless one can argue that if this stabilizing effect is visible
already for the $L^2$-norm it is likely to be more accentuated for the
approximation of averaged quantities where the effect of fluctuations
is expected to be less important due to cancellation.

\subsection{Notation and technical results}
To reduce the number of non-essential constants we will use the
notation $a \lesssim b$ for $a \leq C b$ where $C$ is an $O(1)$
constant independent of the viscosity or the mesh parameter $h$.

We will use standard notations for most Sobolev spaces and norms, but we will 
not distinguish in the notation between the scalar,
vector and tensor valued cases. The space of divergence free functions
in $\bV$ and $[L^2(\Omega)]^d$ will be denoted
\[
\bV_0 := \{\bv \in \bV : \nabla \cdot \bv = 0\}.
\]
and
\[
\bL_0 := \{\bv \in [L^2(\Omega)]^d: \nabla \cdot \bv = 0\}.
\]
To consider different boundary conditions we add a superscript $0$ on
spaces of functions with zero trace on the domain
boundary, for example
$\bV^0_0$.
To simplify we define the $L^2$-inner product for some subset $X \subset
\mathbb{R}^s$, $s=1,\hdots,d+1$, by
\[
(u,v)_X := \int_X u v ~\mbox{d}x, \quad \forall u,v \in L^2(X),
\]
with associated norm $\|\cdot\|_X := (\cdot,\cdot)^{\frac12}_X$.
With some abuse of notation we will not distinguish between the inner
product of scalar, vector or tensor valued functions. The vector
valued case is defined by
\[
(\bu,\bv)_X := \int_X \bu \cdot \bv ~\mbox{d}x, \quad \forall \bu,\bv \in [L^2(X)]^d,
\]
with $\bu \cdot \bv := \sum_{i=1}^d \bu_i \bv_i$
and the tensor valued case is defined by
\[
(\bxi, \bzeta)_X := \int_X \bxi : \bzeta ~\mbox{d}x, \quad \forall \bxi, \bzeta\in
[L^2(X)]^{d\times d},
\]
where $\bxi : \bzeta := \sum_{i,j=1}^d \bxi_{ij} \bzeta_{ij}$. The tensor valued $L^p$ norm will be
defined by
\begin{equation}\label{eq:Lpdef}
\|\bxi\|^p_{[L^p(X)]^{d\times d}}:= \int_X |\bxi|_F^p ~\mbox{d}
x,\quad p \in [2,3], \, \bxi \in [L^p(X)]^d.
\end{equation}
The norm on $C^0(\bar \Omega)$ will be denoted,
\[
\|v\|_\infty := \sup_{\bx \in \bar \Omega} |v(\bx)|.
\]
We also recall H\"olders inequality and Young's inequality that will
be used repeatedly below. For $p,q \in [1,\infty]$ with $p^{-1} +
q^{-1} = 1$ there holds
\begin{equation}\label{eq:holder}
(u,v)_{X} \leq \|u\|_{L^p(X)} \|v\|_{L^q(X)}, \quad u
\in L^p(X), \; v
\in L^q(X),
\end{equation}
\begin{equation}\label{eq:young}
a b \leq p^{-1} (a/\epsilon) ^p + q^{-1} (\epsilon b)^q,\quad a,b
\in \mathbb{R},\quad \epsilon >0.
\end{equation}
The following Poincar\'e-Friedrichs inequality holds for all $1\leq  p <
\infty$.
\begin{lemma}\label{lem:PF}
Let $1\leq  p <
\infty$ then for all $\bv \in W^{1,p}(\Omega)$ with
$\bv\vert_{\partial \Omega} = 0$, 
\[
\|\bv\|_{W^{1,p}(\Omega)} \leq c \|\nabla \bv\|_{L^p(\Omega)}.
\]
\end{lemma}
We will use the notation $\mathcal{T}$ for a quasi-uniform
tesselation of $\Omega$, consisting of simplices $T$ with diameter
$h_T$.  We also introduce the global mesh parameter $h = \max_{T \in
  \mathcal{T}} h_T$. We will denote the broken $L^2$-scalar product
and norm over the elements
of $\mathcal{T}$ by
\[
(u,v) _{\mathcal{T}} := \sum_{T\in \mathcal{T}} (u,v)_{T}, \quad 
\|v\|_{\mathcal{T}} := (v,v) _{\mathcal{T}}^{\frac12}.
\]
 The set of internal faces of the triangulation
$\mathcal{T}$ will be denoted $\mathcal{F}$ and we define the norm of
a function defined on
the skeleton by
\[
(u,v) _{\mathcal{F}} := \sum_{T\in \mathcal{T}} (u,v)_{\partial
    T \setminus \partial \Omega},\quad \|v\|_{\mathcal{F}} := (v,v) _{\mathcal{F}}^{\frac12}.
\]

The following monotonicity and contintuity results of the p-Laplacian
will be useful for the analysis.

\subsubsection{Properties of the p-Laplacian}
For future reference we recall the following two properties:
\begin{lemma}(Monotonicity)
Let $p \ge 2$ then for all $\bxi,
\,\bzeta \in \mathbb{R}^{d\times d}$,
\begin{equation}\label{eq:monotone}
|\bxi - \bzeta|_F^p \leq 2^{p-1} (|\bxi|_F^{p-2} \bxi - |\bzeta|_F^{p-2}
\bzeta, \bxi-\bzeta)_{\mathbb{R}^{d\times d}}.
\end{equation}
\end{lemma}
\begin{proof}
For the monotonicity in the vectorial case see Glowinski and Marocco \cite[Lemma 5.1]{GM75} and \cite[Lemma
3.30]{DEGGH18}. In the context of the
Smagorinsky type model considered here see \cite[Lemma 8.88]{John16}.
\end{proof}
\begin{lemma}(Continuity)\label{lem:cont}
For all $\xi,\, \eta \in \mathbb{R}^{d\times d}$ there holds
\begin{equation}\label{eq:cont1}
||\bxi|_F \bxi - |\bzeta|_F \bzeta|_F \leq  (|\bxi|_F + |\bzeta|_F) |\bxi -\bzeta|_F.
\end{equation}
\end{lemma}
\begin{proof}
The result is elementary. Adding and subtracting $\bzeta$ we have
\[
||\bxi|_F \bxi - |\bzeta|_F \bzeta|_F = ||\bxi|_F (\bxi - \bzeta) -
(|\bzeta|_F  -|\bxi|_F) \bzeta|_F \leq |\bxi|_F|\bxi - \bzeta|_F +
||\bzeta|_F  -|\bxi|_F| |\bzeta|_F. 
\]
The claim then follows by applying the reverse triangle inequality
$||\bzeta|_F  -|\bxi|_F| \leq |\bzeta  -\bxi|_F$ so that
\[
||\bxi|_F \bxi - |\bzeta|_F \bzeta|_F \leq (|\bxi|_F+|\bzeta|_F) |\bzeta  -\bxi|_F.
\]
\end{proof}
%
It follows from \eqref{eq:Lpdef} that for $\bu, \bv \in [L^3(\Omega)]^d$ there holds
\begin{equation}\label{eq:monotone_L3}
\|\nabla (\bu - \bv)\|^3_{L^3(\Omega)} \leq 4 (|\nabla \bu|_F \nabla
\bu - |\nabla \bv|_F \nabla
\bv,\nabla (\bu - \bv))_\Omega.
\end{equation}
Lemma \ref{lem:cont} also implies the continuity
\begin{align}\nonumber
|(|\nabla \bu|_F \nabla \bu - |\nabla \bv|_F \nabla \bv, \nabla
\bw)_\Omega | &\leq ((|\nabla \bu|_F+|\nabla \bv|_F) |\nabla \bu -\nabla
\bv|_F, |\nabla \bw|_F)_\Omega \\ 
\label{eq:continuous}
&\leq 2 (|\nabla \bu|_F |\nabla \bu -\nabla \bv|_F+|\nabla \bu -\nabla \bv|^2_F, |\nabla \bw|_F)_\Omega.
\end{align}
\section{Perturbation equation and scale separation}

We are interested in estimating the growth of perturbations in \eqref{eq:NS_smag}
quantified as the difference of the solution $\hat \bu$, with
$\delta>0$ and the Navier-Stokes' equation $\bu$, with $\delta=0$. To this end we 
will study the perturbation equation obtained
by taking the difference of the equation with $\delta>0$ and with
$\delta=0$ on weak form. 

Let
$\bldeta = \bu - \hat\bu$, then $\bldeta \in \bV_0$ and using
\eqref{eq:NS_smag} we see that 
\begin{align}\label{eq:pert_smag}
&(\partial_t \bldeta,\bv)_{\Omega} + (\hat\bu \cdot \nabla \bldeta,
\bv)_{\Omega} + (\bldeta \cdot \nabla \bu, \bv)_{\Omega} + (\nu
\nabla \bldeta, \nabla \bv)_{\Omega} 
\\ \nonumber
&\qquad \qquad \qquad + (\hat \nu(\bu) \nabla \bu - \hat
\nu(\hat\bu) \nabla \hat\bu, \nabla \bv)_{\Omega} 
= (\hat \nu(\bu) \nabla \bu, \nabla \bv)_{\Omega}
\end{align}
for all $\bv$ such that $\nabla \cdot \bv = 0$.

A classical energy estimate can be obtained from this equation by
testing with $\bldeta$. Clearly 
\[
 (\hat\bu \cdot \nabla \bldeta,
\bldeta )_{\Omega} + (\bldeta \cdot \nabla \bu, \bldeta )_{\Omega} = (\bldeta \cdot \nabla \bu, \bldeta )_{\Omega}
\]
which (as we shall see below) leads to exponential growth with coefficient $\int_I \|\nabla
\bu(t)\|_{L^\infty(\Omega)}$ using Gronwall's estimate. 
The objective of the present contribution is to show that the large
scale velocity that drives the exponential growth can be taken over a
much larger set than $\{\bu,\hat \bu\}$.
Indeed we will introduce a scale separation property defining a set of
fine and coarse scales.
The coarse scales are dependent on the
Smagorsinsky solution implicitly
through the difference between the Navier-Stokes' solution and the
Smagorinsky solution. Under the scale separation assumption, the exponential growth of perturbations will only
depend on the coarse scales. Any exact
solution may be decomposed in a large scale component $\bar \bu$ and a
fine scale fluctuation $\bu'$ so that
\begin{equation}\label{eq:scale_separation}
\begin{cases}
\bu = \bar \bu + \bu', 
\\
\bar \bu \in
[L^1(I;W^{1,\infty}(\Omega))]^d,
\\
\bu' \in \{
[L^{3}(Q)]^d\, | \, \int_Q ((\mu + \hat \nu(\bldeta))^\frac12 -|\bu'| \tau_{L}^{\frac12}) \phi \ge
                                0, \; \forall \phi \in
                                L^{\frac32}(Q), \phi\ge 0\}.
\end{cases}
\end{equation}
where $\tau_L$ is a characteristic time scale of the large scales of
the flow defined by
\begin{equation}\label{eq:char_time}
\tau_{L} := (\alpha(T) + T^{-1})^{-1}, \mbox{ with } \alpha(T) := T^{-1} \int_I \|\nabla \bar \bu(t)\|_{L^\infty(\Omega)} ~\mbox{d}t.
\end{equation}
For $\bu \in [L^1(I;W^{1,\infty}(\Omega))]^d$ the relation always holds for the trivial case
$\bu'=0$, leading to the classical exponential growth with coefficient
proportional to the full gradient of the fluid velocity. We say that the
flow is {\em scale separated} if the decomposition
\eqref{eq:scale_separation} exists with $\tau_L  \sim T$. 
\begin{remark}
Observe that the decomposition depends on
$\hat \nu(\bldeta)$ pointwise and that this quantity can be
very large for rough flows, since $\nabla \bu$, which is not regularized,
may be large locally without violating the a priori regularity assumption. This is a favorable property: as the flow gets more
rough the bound defining admissible small scales is relaxed and the
characteristic time of the resolved scales can increase, leading to
a reduction in the exponential growth of perturbations.
\end{remark}
The rationale for the decomposition is
that the constant in the estimates below will
have exponential growth depending only on $T/\tau_L$.
This shows that small amplitude, high frequency perturbations of 
scale separated flows will not grow exponentially, even if $\|\nabla \bu\|_{L^\infty(\Omega)}$ is large. The
following Lemma is a key tool to quantify this observation.

\begin{lemma}\label{lem:conv_term}
Let $\hat \bu \in \bV$ be the solution of \eqref{eq:NS_smag} with $\delta>0$ and let
$\bu \in \bV$ be the solution with $\delta=0$. Assume that $\bu$ satisfies the scale
separation property \eqref{eq:scale_separation}. Then, denoting
$\bldeta=\bu - \hat \bu$, there holds, for
all $\epsilon > 0$,
\[
 ((\bldeta\cdot \nabla ) \bu, \bldeta)_{Q} \leq   
 \frac{\epsilon}{ 2}\|\mu^{\frac12} \nabla \bldeta\|_{Q}^2
 + \frac{\epsilon}{ 2} \nu(\delta) \|\nabla
 \bldeta\|_{L^3(Q)}^{3} + \||(\epsilon^{-1} \tau_L^{-1}+|\nabla \bar \bu|_F)^{\frac12} \bldeta\|^2_{Q}.
\]
\end{lemma}
\begin{proof}
Using the divergence theorem and the scale separation property we see that
\[
 ((\bldeta \cdot \nabla )\bu, \bldeta)_{Q} =  (\bar \bu + \bu', (\bldeta \cdot
 \nabla ) \bldeta)_{Q}  = ((\bldeta \cdot \nabla )\bar \bu, \bldeta)_{Q}+ (\bu', (\bldeta \cdot
 \nabla ) \bldeta)_{Q}.
\]
Using the bound of \eqref{eq:scale_separation} we see that, since by
assumption $ (\bldeta \cdot
 \nabla ) \bldeta \in [L^{\frac32}(Q)]^d$,
\[
|(\bu', (\bldeta \cdot
 \nabla ) \bldeta)_{Q}| \leq |(\tau_{L}^{-\frac12} |\bldeta|, (\mu+\hat
 \nu(\bldeta))^{\frac12}|
 \nabla  \bldeta|)_{Q}|.
\]
By the definition of $\hat\nu(\bldeta)$
\[
\|\hat\nu(\bldeta)^{\frac12} \nabla \bldeta\|_{Q} =
\left(\int_Q \nu(\delta) |\nabla \bldeta|_F^{3}
  ~\mbox{d}x \mbox{d}t\right)^{\frac12} = \| \nu(\delta)^{\frac13} \nabla \bldeta\|_{L^3(Q)}^{\frac32}.
\]
An application of the Cauchy-Schwarz inequality followed by the inequality \eqref{eq:young}, with $p=q=2$ then shows that
\begin{align*}
 |(\tau_{L}^{-\frac12} |\bldeta|, (\mu+\hat
 \nu(\bldeta))^{\frac12}
 \nabla  \bldeta|)_{Q}| 
 &\leq  \|\tau_L^{-\frac12}
 \bldeta\|_Q\|\hat\nu(\bldeta)^{\frac12} \nabla
 \bldeta\|_Q 
 \\
 &\leq  \frac{1}{ 2 \epsilon} \|\tau_L^{-\frac12}
 \bldeta\|_Q^2 +  \frac{\epsilon}{ 2} \|\mu^{\frac12} \nabla \bldeta\|^2_{L^2(Q)}+ \frac{\epsilon}{ 2} \|\nu(\delta)^{\frac13} \nabla \bldeta\|^3_{L^3(Q)}.
\end{align*}
Collecting the above inequalities we have
\begin{align*}
((\bldeta \cdot \nabla )\bu, \bldeta)_{Q} &\leq | ((\bldeta \cdot \nabla )\bar \bu, \bldeta)_{\Omega}| + ( |\bldeta|,|\bu'|
 |\nabla \bldeta|)_{Q} 
 \\
&\leq  (\bldeta |\nabla \bar \bu|, \bldeta)_{Q} +
  \|\tau_L^{-\frac12} \bldeta\|_{Q}
 (\|\hat\nu(\bldeta)^{\frac12} \nabla \bldeta\|_{Q} + \|\mu^{\frac12} \nabla \bldeta\|_{Q})
 \\
&\leq  (\bldeta |\nabla \bar \bu|, \bldeta)_{Q} +  \frac{1}{\epsilon}
\|\tau_L^{-\frac12} \bldeta\|_{Q}^2 + \frac{\epsilon}{ 2}\|\mu^{\frac12} \nabla \bldeta\|_{Q}^2+
 \frac{\epsilon}{ 2} \|\nu(\delta)^{\frac13} \nabla \bldeta\|^3_{L^3(Q)}.
\end{align*}
The result follows by noting that
\[
(\bldeta |\nabla \bar \bu|, \bldeta)_{Q} + \epsilon^{-1}\|\tau_L^{-\frac12}
\bldeta\|_{Q}^2 = \||(\epsilon \tau_L^{-1}+|\nabla \bar \bu|_F)^{\frac12} \bldeta\|^2_{Q}.
\]
\end{proof}

\section{Perturbation growth on the continuous level for scale
  separated flows}\label{sec:scale_pert}
We will first prove, using a perturbation argument on the continuous
equations, that the perturbation induced by the Smagorinsky term allows
for an $O(\nu(\delta)^{\frac12})$ error estimate in the $L^2$-norm between the Navier-Stokes' solution
and the Navier-Stokes-Smagorinsky solution. In case the solution is
smooth this bound can be improved to $O(\nu(\delta))$.
\begin{theorem}\label{thm:red_exp}
Let $\bu \in \bV$ be the solution to
\eqref{eq:NS_smag} with $\delta=0$ and let $\hat \bu$ be the solution of
\eqref{eq:NS_smag} with $\delta \ge 0$. Then there holds, with $\bldeta:=\bu -\hat \bu$
\[
\sup_{t \in I} \|\bldeta(t)\|^2_{\Omega}  +\|\mu^{\frac12} \nabla \bldeta\|_{Q}^2 + 
\frac{1}{6}\|\nu(\delta)^{\frac13} \nabla \bldeta\|_{L^3(Q)}^3
\leq  e^{14 \frac{T}{\tau_{L}}}  \frac83 \nu(\delta) \|\nabla
\bu\|_{L^3(Q)}^{3},
\]
where $\tau_{L}$ is defined by \eqref{eq:char_time}.
If in addition $\nabla \cdot(\hat \nu(\bu) \nabla) \bu \in [L^2(I;L^2(\Omega))]^d$
then there holds
\begin{equation*}
\sup_{t \in I} \|\bldeta(t)\|^2_{\Omega}  +\|\mu^{\frac12} \nabla \bldeta\|_{Q}^2 + 
\frac{1}{6} \|\nu(\delta)^{\frac13} \nabla \bldeta\|_{L^3(Q)}^3 
 \leq
e^{15\frac{T}{\tau_{L}}} \tau_{L} \nu(\delta)^2 \|\nabla \cdot(|\nabla \bu|\nabla \bu)\|_{Q}^{2}.
\end{equation*}
\end{theorem}
\begin{proof}
Taking $\bv = \bldeta$ in \eqref{eq:pert_smag} we see that, using the skew symmetry of the
convective term and the monotonicity of the p-Laplacian, Lemma \ref{eq:monotone},
we get
\begin{equation}\label{eq:pert_eq_start}
\frac12 \frac{d}{dt} \|\bldeta\|_{\Omega}^2 + ((\bldeta \cdot \nabla)
\bu, \bldeta)_{\Omega} + \|\mu^{\frac12}  \nabla \bldeta\|^2_{\Omega} + \frac14
\|\nu(\delta)^{\frac13} \nabla \bldeta (t)\|_{L^3(\Omega)}^3 \leq  (\hat \nu(\bu) \nabla \bu, \nabla \bldeta)_{\Omega}.
\end{equation}
We can bound
the right hand side using H\"olders inequality \eqref{eq:holder} and
Young's inequality \eqref{eq:young}, with $p=3/2$ and $q=3$,
\begin{align}\nonumber
 (\hat \nu(\bu) \nabla \bu, \nabla \bldeta)_{\Omega} &\leq \|\hat \nu(\bu)
 \nabla \bu\|_{L^{\frac32}(\Omega)} \|\nabla \bldeta\|_{L^{3}(\Omega)}
\\ \nonumber
& \leq \nu(\delta)^{\frac{2}{3}} \left(\frac14\right)^{-\frac13} \||\nabla \bu|_F
 \nabla \bu\|_{L^{\frac32}(\Omega)}  \left(\frac14\right)^{\frac13} \nu(\delta)^{\frac{1}{3}} \|\nabla
 \bldeta\|_{L^{3}(\Omega)}
 \\ \nonumber
&\leq \frac23 \left(\frac14\right)^{-\frac12} \nu(\delta) \|\nabla
\bu\|_{L^3(\Omega)}^{3}+\frac13 \left(\frac14\right) \nu(\delta) \|\nabla
 \bldeta\|^3_{L^{3}(\Omega)}
\\ \label{eq:consist_error1}
&\leq \frac43 \nu(\delta) \|\nabla
\bu\|_{L^3(\Omega)}^{3}+\frac{1}{12} c_p \|\nu(\delta)^{\frac13}\nabla
 \bldeta\|^3_{L^{3}(\Omega)}.
\end{align}
The second term in the right hand side is now absorbed by the fourth
term in the left hand side of \eqref{eq:pert_eq_start}. After
integration in time we obtain
\[
\frac12 \|\bldeta(T)\|_{\Omega}^2 + \|\mu^{\frac12}  \nabla
\bldeta\|^2_{Q} + \frac16 
\|\nu(\delta)^{\frac13}\nabla \bldeta (t)\|_{L^3(Q)}^3 \leq
\frac43 \nu(\delta)\|\nabla
\bu\|_{L^3(Q)}^{3} -   ((\bldeta \cdot \nabla)
\bu, \bldeta)_{Q}.
\]
Applying Lemma \ref{lem:conv_term} we proceed to bound the second term
in the right hand side which is the last term that does not have sign.
\begin{align*}
&\frac12 \|\bldeta(T)\|_{\Omega}^2 + \|\mu^{\frac12}  \nabla
\bldeta\|^2_{Q} + \frac16 
\|\nu(\delta)^{\frac13}\nabla \bldeta (t)\|_{L^3(Q)}^3 
\\
&\qquad \qquad \leq
\frac43 \delta^2 \|\nabla
\bu\|_{L^3(Q)}^{3} +  \|(\epsilon^{-1} \tau_L^{-1}+|\nabla \bar \bu|_F)^{\frac12} \bldeta\|_{Q}^2 +
 \frac{\epsilon}{ 2} \|\nu(\delta)^{\frac13} \nabla
 \bldeta\|^3_{L^3(Q)} +\frac{\epsilon}{ 2} \|\mu^{\frac12} \nabla
 \bldeta\|^2_{Q}.
\end{align*}
Taking $\epsilon = 1/6$ and multiplying through with $2$ we obtain the
bound
\begin{equation*}
\|\bldeta(T)\|_{\Omega}^2 + \|\mu^{\frac12}  \nabla
\bldeta\|^2_{Q} + \frac{1}{6}
\|\nu(\delta)^{\frac13}\nabla \bldeta (t)\|_{L^3(Q)}^3  \leq \frac83
\nu(\delta) \|\nabla \bu\|_{L^3(Q)}^{3}
+  2 \|(6 \tau_L^{-1}+|\nabla \bar \bu|_F)^{\frac12} \bldeta\|_{Q}^2.
\end{equation*}
Rewriting this as
\begin{align*}
&\|\bldeta(T)\|^2_{\Omega} +  \|\mu^{\frac12}  \nabla
\bldeta\|^2_{Q}~\mbox{d}t+ \frac{1}{6}  \|\nu(\delta)^{\frac13}\nabla
\bldeta\|^3_{L^{3}(Q)} 
\\
&\qquad \qquad \leq 2 \int_I (6 \tau_L^{-1}+\|\nabla \bar \bu(t)\|_{L^\infty(\Omega)})\|\bldeta(t)\|_{\Omega}^2
~\mbox{d}t + \frac83  \nu(\delta) \|\nabla
\bu\|_{L^3(Q)}^{3},
\end{align*}
we conclude by applying Gronwall's inequality, leading to
\[
\|\bldeta (T)\|^2_{\Omega} +  \|\mu^{\frac12}  \nabla
\bldeta\|^2_{Q} ~\mbox{d}t+\frac{1}{6} \|\nu(\delta)^{\frac13} \nabla \bldeta\|^3_{L^{3}(Q)} 
\leq e^{14 \frac{T}{\tau_{L}}}  \frac83\nu(\delta) \|\nabla
\bu\|_{L^3(Q)}^{3} ,
\]
where we used that $2 \int_I (6 \tau_L^{-1}+\|\nabla \bar
\bu\|_{\infty})~\mbox{d}t \leq 14 T/\tau_{L}$. 
This proves the first inequality. To prove the second observe that
proceeding by integration by parts we have instead of
\eqref{eq:consist_error1},
\begin{equation}\label{eq:consist_error2}
(\hat \nu(\bu) \nabla \bu, \nabla \bldeta)_{\Omega} \leq (\nabla
\cdot \nu(\bu) \nabla \bu, \bldeta)_{\Omega} \leq \tau_{L} \nu(\delta)^2
\|\nabla
\cdot |\nabla \bu| \nabla \bu\|_{\Omega}^2 + \tau_{L}^{-1}
\|\bldeta\|^2_{\Omega}.
\end{equation}
The conclusion once again follows using Gronwall's Lemma.
\end{proof}
\begin{remark}
Writing out the bound on $\bu'$ of \eqref{eq:scale_separation}
pointwise, with $\nu(\delta) = \delta^2$ and $\mu=0$, leads
to
\[
|\bu'|^2 \tau_{L} \leq \delta^2 |\nabla (\bu -\hat \bu)|_F.
\]
From this it follows that increasing $\delta$ must have one (or a
combination) of the
following consequences:
\begin{enumerate}
\item the error in $\nabla (\bu -\hat \bu)$ decreases, the
  scale separation stays the same;
\item the characteristic time $\tau_{L}$ increases and as a consequence the exponential growth
  decreases;
\item the small scale $|\bu'|$ increases; which implicitly allows for
  a decrease in $\bar \bu$ through the definition of the scale
  separation and a possible increase in $\tau_{L}$.  
\end{enumerate}
All these three possibilities point to the fact that either the error
is reduced by the increase of the parameter, or the exponential coefficient in the
estimate will decrease leading to a decrease in the exponential
growth. This effect can offset the effect of the increased consistency
introduced by increasing $\delta$. Another salient conclusion is that
if $\delta$ is coupled to the mesh size, there may be flow
configurations where the computational error grows under mesh
refinement. Although the consistency error decreases and numerical
resolution increases, the set of
non-essential fine scale decreases leading to increased exponential
growth of perturbations. This hints at a resolution barrier for
Smagorinsky LES
beyond which a full DNS is required to enhance accuracy further.
\end{remark}
The result of Theorem
\ref{thm:red_exp}
shows that the exponential growth of perturbations can be moderated by
the Smagorinsky term and hence the turbulence model indeed has a
stabilizing effect, which was the first objective of the present work.
\section{The Smagorinsky model as a numerical stabilizer}
In this section we will consider the situation where the
Smagorinsky term is not a physical model, but a stabilizing term in a
numerical method. We will consider a piecewise affine finite
element method that fits in a discrete de Rham
complex, see for instance \cite{Zhang08, GN18, CH18}. That means that the space has a divergence free subspace with
optimal approximation properties. It turns out that the
discrete de Rham complex and piecewise affine approximation are
exactly the properties that make the
Smagorinsky model a stabilizing term, with well balanced stability
versus accuracy. The affine approximation order makes the
consistency error of the nonlinear viscosity similar to the accuracy
of the approximation and the additional stability obtained through the
exact satisfaction of the divergence free condition reduces the need
of stabilization so that the second order Smagorinsky term is
sufficiently large to control fluctuations inside elements. To counter
instabilities due to the lack of $C^1$-continuity of the
approximation space a penalty term is added on a certain component of
the jump of the streamline derivative. Observe that this latter
stabilizing term echos the early approach to stabilization of
divergence free elements proposed in \cite{BL08}. In this work however no
improved convergence rate was obtained as a consequence of the
stabilization. \emph{To the best of our knowledge, the analysis below gives the first Reynolds number
robust $L^2$-error estimate with $O(h^{\frac32})$ convergence for a
piecewise affine $H^1$-conforming finite element method, satisfying the
incompressibility constraint exactly.} The analysis draws on recent
results using Galerkin-Least Squares
stabilization of the vorticity equation \cite{ABBGLM20}. The analysis
presented there however does not carry over to the piece affine case
considered herein. For the corresponding result
using $H_{div}$-conforming methods
we refer to \cite{BBG20}.


\subsection{Approximation space and technical results}
For the purposes of the
present paper it is sufficient to know that the space allows for a
divergence free subspace with optimal approximation properties in
$H^1$ and $L^2$. The below analysis also uses that the space is affine
to achieve the optimal error estimate. For higher order
spaces the weak consistency of the Smagorisky model is insufficient,
even if the flow is laminar. We let $\bV_h $ denote the velocity subspace of piecewise
affine vector functions in $[H^1(\Omega)]^d$ constructed such that it
satisfies a discrete inf-sup condition for the divergence free
constraint using some pressure space $Q_h$ such that $\nabla \cdot \bV_h \in Q_h$. Observe that by working in the divergence free space the
pressure can be eliminated in the analysis, which we will make use of
to reduce the technical detail below and we will therefore not discuss
$Q_h$ further but refer to \cite{CH18,GN18}. Below we will assume that
$\Omega \subset \mathbb{R}^3$ is convex, simply connected,
polyhedron. The two-dimensional case is also covered by the analysis,
but to reduce notation we omit it from the discussion.

The
 following classical inverse and trace inequalities are used
 frequently in the analysis:
\begin{itemize}
\item Inverse 
inequalities,
\begin{equation}\label{inverse}
|v|_{H^1(T)} \lesssim h^{-1}_T \|v\|_{L^2(T)}\quad \forall v \in \mathbb{P}_1(T).
\end{equation}
Here $\mathbb{P}_1(T)$ denotes the set of polynomials of
degree less than or equal to $1$ on the simplex $T$.
For $p\ge q\ge 1$, $l=0,1$, there holds
\begin{equation}\label{eq:Lpinv}
\|\bv\|_{W^{l,p}(\Omega)} \leq C h^{\frac{d}{p}-\frac{d}{q}}
\|\bv\|_{W^{l,q}(\Omega)}  \quad \forall \bv \in \bV_h.
\end{equation}
For a proof of \eqref{inverse} see \cite[Section 1.4.3]{DiPE12} and
for \eqref{eq:Lpinv} see
\cite[Corollary 1.141]{EG04}.
\item Trace inequalities (see \cite[Section 1.4.3]{DiPE12}),
\begin{equation}\label{trace_H1}
\|v\|_{L^2(\partial T)} \leq C \left( h_T^{-\frac12} \|v\|_{L^2(T)} +
h_T^{\frac12} \|v\|_{H^1(T)}\right)\quad \forall v \in H^1(T).
\end{equation}
\end{itemize}

\subsubsection{Approximation error estimates}
To simplify the analysis we introduce the divergence free space
with a homogeneous Dirichlet condition on the normal component,
\[
\bV_{0\bn} := \{\bv \in \bV:\, \nabla \cdot \bv = 0, \bv \cdot \bn =0
\mbox{ on } \partial \Omega \}.
\]
we also define the divergence free subspace of $\bV_h$
\[
\bV_{0h} := \bV_h \cap \bV_{0\bn}.
\]
We introduce the $L^2$-orthogonal projections $\pi_h:[L^2(\Omega)]^d\mapsto \bV_{0h}$
and $\bPi_h: \bL_0 \mapsto \bV_{0h}$. By the quasi
uniformity of the mesh the following bounds hold using standard finite element
approximation arguments
\begin{equation}\label{eq:L2proj1}
\|\bu - \pi_h \bu\|_{L^p(\Omega)} + h \|\nabla (\bu - \pi_h \bu)\|_{L^p(\Omega)}
\leq C h |\bu|_{W^{2,p}(\Omega)}, \quad p \ge 1, \quad \forall \bu \in
\bV_{0\bn} \cap [W^{2,p}(\Omega)]^d,
\end{equation}
and
\begin{equation}\label{eq:L2proj2}
\|\bu - \bPi_h \bu\|_\Omega + h \|\nabla (\bu - \bPi_h \bu)\|_\Omega+
h^{-\frac12} \|\bu - \bPi_h \bu\|_{\partial \Omega}
\leq C h |\bu|_{H^2(\Omega)}\quad \forall \bu \in \bV_{0\bn} \cap [H^2(\Omega)]^d.
\end{equation}
We also recall the $L^p$-stability of $\pi_h$, for $p \ge 1$ there
holds \cite{CT87},
\begin{equation}\label{eq:Lpstab}
\|\pi_h \bv\|_{L^p(\Omega)} \leq \| \bv\|_{L^p(\Omega)}  \quad \mbox{and
}  \quad \|\nabla \pi_h \bv\|_{L^p(\Omega)} \leq \|\nabla \bv\|_{L^p(\Omega)}.
\end{equation}
We will use the notation $\jump{\cdot}$ for the jump of a quantity
across an element boundary. In particular, we define
\begin{equation}\label{eq:jump_def}
\jump{\nabla v }:=\nabla v\vert_{T_1} \cdot\bn_1 + \nabla
v\vert_{T_2} \cdot \bn_2,\qquad  \jump{\bv \times \bn} := \bv\vert_{T_1} \times \bn_1 + \bv\vert_{T_2} \times \bn_2,
\end{equation}
to be the jump over the face $F:= \bar T_1 \cap \bar T_2$, where
$\bn_i$ is the outward pointing normal of the element $T$. The jump of
the gradient tensor $\jump{\nabla \bv }$ is defined by
applying the left inequality of \eqref{eq:jump_def} to each column vector. Observe
that it is an immediate consequence of \eqref{trace_H1} and
\eqref{eq:L2proj2} that for all $\bv \in [H^2(\Omega)]^d$,
\begin{equation}\label{eq:jump_bound}
\|h \jump{\nabla \bPi \bv}\|_{\mathcal{F}} \lesssim  h^{\frac32} |\bv|_{H^2(\Omega)}.
\end{equation}
To see this, note that
using the regularity of $\bv$ and a trace inequality \eqref{trace_H1} on each face $F$
shared by simplices $T_1$ and $T_2$, we have
\[
\|\jump{\nabla\bPi\bv}\|_{F} = \|\jump{\nabla(\bPi\bv - \bv)}\|_{F}
\leq C (h^{-\frac12} \|\nabla(\bPi\bv - \bv)\|_{T_1\cup T_2 }+
h^{\frac12} (|\bv|_{H^2(T_1)}+|\bv|_{H^2(T_2)})).
\]
 The following Lemma is an immediate consequence of \eqref{eq:Lpstab}.
\begin{lemma}\label{lem:lptol2}
Let $\bw = \bu + \bv_h$ with $\bu \in [W^{1,p}(\Omega)]^d$ and $\bv_h
\in \bV_h$. Then there holds for $p \ge 2$,
\[
\|\bw\|_{W^{s,p}(\Omega)} \leq \|\bu - \pi_h \bu\|_{W^{s,p}(\Omega)}
+ C h^{\frac{d(2-p)}{2p}} \|\bw\|_{H^s(\Omega)},\quad s=0,1.
\]
\end{lemma}
\begin{proof}
By the triangle inequality followed by an inverse inequality, \eqref{eq:Lpinv} and the stability of the
$L^2$-projection \eqref{eq:Lpstab}, we have
\begin{align*}
\|\bw\|_{W^{s,p}(\Omega)} &= \|\bw - \pi_h \bw\|_{W^{s,p}(\Omega)}
+\|\pi_h\bw\|_{W^{s,p}(\Omega)} 
\\
&\leq \|\bu - \pi_h \bu\|_{W^{s,p}(\Omega)}
+ C h^{\frac{d(2-p)}{2p}} \|\bw\|_{H^s(\Omega)},
\end{align*}
which completes the proof.
\end{proof}

We will now prove some estimates in weaker norms that will be helpful
in the analysis.
We recall the following surjectivity property of the curl
operator on simply connected polyhedral domains $\omega$. For all $\bv \in
\bV_{0\bn}$ there exists $\bvarphi\in \bV_0(\omega) \cap [H^1(\omega)]^d$ such
that $\nabla \times \bvarphi = \bv$ in $\omega$, $\bvarphi \times \bn_\omega = 0$
on $\partial \omega$ and satisfying the stability $\|\nabla
\bvarphi\|_\omega \lesssim \|\nabla \times
\bvarphi\|_\omega$. \cite[Theorem 3.17]{ABDG98}
Using the above properties we now introduce the regularized approximation error $\bE \in \bL_0(\Omega) \cap [H^1(\Omega)]^d$ defined by
\begin{alignat}{3}\label{eq:E1_Omega}
\nabla \times \bE &= \bv - \bPi_h \bv &\qquad & \mbox{in } \Omega,
\\
\nabla \cdot \bE & =0 &\qquad &  \mbox{in } \Omega, 
\\
\bn \times \bE & = 0 &\qquad & \mbox{on } \partial \Omega. \label{eq:E3_Omega}
\end{alignat}

\begin{lemma}\label{lem:weak_approx}
Let $\bE$ be defined by
\eqref{eq:E1_Omega}-\eqref{eq:E3_Omega} then
\begin{equation}\label{eq:E_approx1}
\|\bE\|_\Omega + \|h^{\frac12} \bE\|_{\mathcal{F}} + h \|\nabla \bE\|_\Omega  \lesssim h \|\bv - \bPi_h \bv\|_\Omega \lesssim h^{3} |\bv|_{H^2(\Omega)}.
\end{equation}
\end{lemma}
\begin{proof}

By our assumptions on $\Omega$ there holds $\|\nabla \bE\|_\Omega
\lesssim \|\nabla \times \bE\|_\Omega$ and there exists $\bPsi \in
\bL_0(\Omega) \cap [H^1(\Omega)]^d$  and
$\bPsi \cdot \bn = 0$ on $\partial \Omega$ such that $\nabla \times \bPsi = \bE$, with
$\|\nabla \bPsi\|_\Omega \lesssim \|\nabla \times \bPsi\|_\Omega$  \cite[Theorem 3.12]{ABDG98}.

Using the $\bE$ and $\bPsi$ we may bound the error in the following way
\begin{align*}
\|\bE\|_\Omega^2 &= (\bE , \nabla \times \bPsi)_\Omega  
\\
&= (\bv - \bPi_h \bv
,\bPsi - \bPi_h \bPsi)_\Omega + \underbrace{(\bE, \bPsi \times
  \bn)_{\partial \Omega}}_{ = (\bE \times \bn, \bPsi)_{\partial \Omega}
  = 0}
\\
&\lesssim  h \|\nabla \bPsi\|_\Omega \|\bv - \bPi_h \bv \|_\Omega 
\\
&\lesssim h \|\bE\|_\Omega \|\bv - \bPi_h \bv \|_\Omega.
\end{align*}
We conclude that 
\begin{equation}\label{eq:bulk_E}
\|\bE\|_\Omega \lesssim  h \|\bv - \bPi_h \bv\|_\Omega.
\end{equation}
On each face we use the trace inequality \eqref{trace_H1}
\[
\| \bE\|_F \lesssim h^{-\frac12} \|\bE\|_T + h^{\frac12} \|\nabla \bE\|_{T}.
\]
It follows that
\begin{equation}\label{eq:face_E}
\|\bE\|_{\mathcal{F}}^2 \leq h^{-1} \| \bE\|^2_\Omega + h
|\bE|_{H^1(\Omega)}^2 \lesssim h  \|\bv - \bPi_h \bv\|_\Omega^2.
\end{equation}
Multiplying \eqref{eq:face_E} by $h$ and using approximation in the
right hand sides of \eqref{eq:bulk_E} and \eqref{eq:face_E}
leads to the desired inequality.
\end{proof}
\subsubsection{Vector identity}
For the analysis below the following elementary vector identity
\cite{BBG20} will
be useful. For two $3 \times 3$ matrices $A$ and $B$ with rows $A_i$
and $B_i$ ($i=1,2,3$) we define the vector quantity
 $\bc:=A \times B$ with the components $c_1= A_2 \cdot B_3-A_3 \cdot B_2$, $c_2= -(A_1
 \cdot B_3-A_3 \cdot B_1)$ $c_3= A_1 \cdot B_2-A_2 \cdot B_1$. A
 simple calculation then gives the following identity.  
\begin{lemma}\label{derivative-identity}
 For sufficiently smooth vectors $\bbeta$ and $\bv$ there holds
 \begin{equation*}
 \curl( \bbeta \cdot \nabla \bv)= \bbeta \cdot \nabla (\curl \bv)+ ((\nabla \bbeta)^t \times \nabla \bv) . 
 \end{equation*}
 \end{lemma}

\subsection{A linear model problem}
Since the approach to stabilization in the present work is non-standard we first consider the linear 
model problem for inviscid flow introduced in \cite{BBG20}. The Smagorinsky bulk term is not strictly necessary
for the analysis
in the linear case. Nevertheless we keep this term, since in this
simplified context the stabilizing mechanisms become clear. The model problem takes the form, find a velocity $\bu$ and a pressure $p$ satisfying
\begin{subequations}\label{pde}
\begin{alignat}{2}
\nabla \cdot  (\bu \otimes \bbeta) + \sigma \bu+ \nabla p = &\bff \quad  && \text{ in }  \Omega, \\
\nabla \cdot \bu=&0  \quad  && \text{ in } \Omega,  \\
\bu \cdot \bn = & 0   \quad && \text{ on } \partial \Omega.
\end{alignat}
\end{subequations}
We think of $\bu$ and $\bbeta$ as column vectors and we set $\bu
\otimes \bbeta = \bu \bbeta^t$. We assume that $\bbeta \in
[C^{1}(\bar \Omega)]^d \cap \bV_{0\bn}$ and $\sigma\in
\mathbb{R}_+$. To ensure uniqueness we assume $\sigma$
sufficiently big compared to $\|\nabla \bbeta\|_{\infty}$, for details
see \cite{BBG20}.

The numerical method we will analyze here reads: Find $\bu_h \in \bV_{0h}$ such that
\begin{subequations}\label{linearfem}
\begin{alignat}{2}
-(\bu_h, \bbeta \cdot \nabla \bv_h)_\Omega+(\sigma \bu_h,
\bv_h)_\Omega +\gamma s(\bu_h,\bv_h)=&\, (\bff,\bv_h)_\Omega \quad && \forall \bv_h \in \bV_{0h}
\end{alignat}
\end{subequations}
where $\gamma>0$ is a dimensionless parameter.
We define the stabilizing operator as a combination of a face penalty
operator and an artificial viscosity term in the bulk. In the fully
nonlinear case the bulk viscosity will be replaced by the Smagorinsky model,
\[
 s(\bu_h,\bv_h):= (\delta^2 |\nabla \bbeta|_F \nabla \bu_h,\nabla
 \bv_h)_\Omega+ ( h^{2}|\bbeta|^{-1}
 \jump{(\bbeta \cdot \nabla )\bu_h\times \bn} ,
 \jump{(\bbeta \cdot \nabla) \bv_h \times \bn})_{\mathcal{F}}.
\]
We define the stabilization semi-norm by $|\bv|_s :=
s(\bv,\bv)^{\frac12}$ and note that the following approximation
estimate holds
\begin{lemma}\label{lem:stab_approx}
Let $\bv \in [H^2(\Omega)]^d$ and $\bbeta \in [C^{1}(\bar \Omega)]^d$ then
\[
|\bv - \bPi \bv|_s  \lesssim (\delta \|\nabla
\bbeta\|_{\infty}^{\frac12}+ h^{\frac12} \|
\bbeta\|^{\frac12}_{\infty})
h |\bv|_{H^2(\Omega)}.
\]
If $\|\nabla
\bbeta\|_{\infty}\lesssim h^{-1} \|
\bbeta\|_{\infty}$, $\gamma = O(1)$ and $\delta = O(h)$ then
\[
|\bv - \bPi \bv|_s  \lesssim C \|
\bbeta\|^{\frac12}_{\infty}
h^{\frac32} |\bv|_{H^2(\Omega)}.
\]
\end{lemma}
\begin{proof}
First note that for the linearized Smagorinsky term there holds by \eqref{eq:L2proj2}
\[
\|\delta |\nabla \bbeta|_F^{\frac12} \nabla (\bv - \bPi \bv)\|_\Omega \leq
\delta \|\nabla \bbeta\|_\infty^{\frac12} \|\nabla (\bv - \bPi \bv)\|_\Omega
\leq \delta \|\nabla \bbeta\|_\infty^{\frac12} h |\bv|_{H^2(\Omega)}.
\]
For the term on the faces on the other hand we have
\[
\gamma^{\frac12}\||\bbeta|^{-\frac12} h \jump{(\bbeta \cdot \nabla
  )\bPi\bv\times \bn}\|_{\mathcal{F}} \leq \gamma^{\frac12} h
\|\bbeta\|^{\frac12}_\infty \|\jump{\nabla\bPi\bv}\|_{\mathcal{F}} .
\]
The first claim now follows by summing over all faces, using
\eqref{eq:jump_bound}. The second claim is immediate from the first claim and the assumptions.
\end{proof}
Clearly by taking $\bv_h = \bu_h$ in
\eqref{linearfem}, integrating by
parts and using the properties of $\bbeta$ and a Cauchy-Schwarz
inequality we have the a priori
estimate
\begin{equation}\label{linear_apriori}
\|\sigma^{\frac12} \bu_h\|_\Omega \leq \sigma^{-\frac12} \|\bff\|_\Omega.
\end{equation}
Thanks to equation ~\eqref{linear_apriori} the problem \eqref{linearfem} has a unique solution. Moreover, the
method~\eqref{linearfem}  is weakly consistent; in fact,  for $(\bu,p)
\in [L^2(\Omega)]^d \times L^2_0(\Omega)$ solving \eqref{pde} we have,
using that $(p,\nabla \cdot  \bv_h)_\Omega= 0$ for all $\bv_h \in \bV_{0h}$,
\begin{subequations}\label{weak}
\begin{alignat}{2}
-(\bu, \bbeta \cdot \nabla \bv_h)_\Omega+(\sigma \bu, \bv_h)_\Omega
=&\, (\bff,\bv_h)_\Omega\quad && \forall \bv_h \in \bV_{0h}.
\end{alignat}
\end{subequations}
Following the ideas of \cite{BBG20} we now prove an a priori error estimate
in the $L^2$-norm for the finite element solution to
\eqref{linearfem}.
\begin{proposition}\label{prop:lin_error_est}
Let $\bu \in \bL_0 \cap[H^2(\Omega)]^d$ be the solution to
\eqref{pde},  and $\bu_h \in \bV_{0h}$ the solution to
\eqref{linearfem}. Also assume that $\|\nabla
\bbeta\|_{\infty} \lesssim \delta^{-1} \|
\bbeta\|_{\infty}$, $\gamma = O(1)$ and $\delta = O(h)$. Then there holds
\[
\sigma^{\frac12} \|\bu - \bu_h\|_\Omega + |\bu- \bu_h|_s
\lesssim h^{\frac32} (\|\bbeta\|^{\frac12}_{\infty}
+ \sigma^{\frac12} h^{\frac12}) |\bu|_{H^2(\Omega)} + \sigma^{-1} h^{\frac32} \||\nabla
\bbeta| \nabla \bu\|_{H^1(\Omega)}.
\]
\end{proposition}
\begin{proof}
Let $\bldeta= \bu-\bu_h$ and observe that since by definition $\nabla \cdot \bbeta = 0$ and $\bbeta \cdot \bn =
0$ on $\partial \Omega$
\[
\|\sigma^{\frac12} \bldeta\|^2_\Omega + |\bldeta|_s^2=-(\bldeta, \bbeta
\cdot \nabla \bldeta)_\Omega+(\sigma \bldeta, \bldeta)_\Omega+ s(\bldeta,\bldeta).
\]
We can then
use \eqref{linearfem} and \eqref{weak} and the
divergence theorem to
obtain
\begin{align}\label{eq:lin_pert}
\|\sigma^{\frac12} \bldeta\|^2_\Omega + \gamma |\bldeta|_s^2 
&=(\bbeta \cdot \nabla\bldeta, \bu -
\bPi \bu)_\Omega+(\sigma \bldeta, \bu -
\bPi \bu)_\Omega 
\\ \nonumber
&\qquad 
+\gamma  s(\bu,\bldeta)- \gamma s(\bu,\bu -
\bPi \bu)+\gamma s(\bldeta,\bu -\bPi \bu).
\end{align}
For the second term on the right hand side we have 
\begin{equation}\label{eq:reac_term}
(\sigma \bldeta, \bu -
\bPi \bu)_\Omega \leq \|\sigma^{\frac12}(\bu -
\bPi \bu)\|^2_\Omega,
\end{equation}
and for the last three terms
\begin{align*}
& s(\bu,\bldeta)-s(\bu,\bu -
\bPi \bu)+s(\bldeta,\bu -\bPi \bu) 
\\
&\qquad  = -(\delta^2 \nabla \cdot |\nabla
\bbeta| \nabla \bu, \bldeta - \bu -
\bPi \bu) + (\delta^2 |\nabla
\bbeta| \nabla \bu\cdot \bn, \bldeta - \bu -
\bPi \bu)_{\partial \Omega} +s(\bldeta,\bu -\bPi \bu).
\end{align*}
It follows using the Cauchy-Schwarz inequality and the
arithmetic-geometric inequality that
\begin{equation}\label{eq:smag_term}
 -(\delta^2 \nabla \cdot |\nabla
\bbeta| \nabla \bu, \bldeta - \bu -
\bPi \bu) \leq 2 \sigma^{-1}  \|\delta^2 \nabla \cdot |\nabla
\bbeta| \nabla \bu\|_\Omega^2 + \frac14 \sigma \|\bldeta\|^2_\Omega +
\frac14 \sigma \| \bu -
\bPi \bu\|^2_\Omega,
\end{equation}
and
\begin{equation}\label{eq:stab_term}
s(\bldeta,\bu -\bPi \bu) \leq \frac14 |\bldeta|_s^2 + |\bu - \bPi
\bu|_s^2.
\end{equation}
For the boundary term we use the Cauchy-Schwarz inequality followed by
a global and local trace inequality and approximation to obtain
\begin{multline*}
(\delta^2 |\nabla
\bbeta| \nabla \bu\cdot \bn, \bldeta - \bu -
\bPi \bu)_{\partial \Omega} \leq C  h^{\frac32} \||\nabla
\bbeta| \nabla \bu\|_{H^1(\Omega)} h^{\frac12}(\| \bldeta\|_{\partial \Omega} + \|\bu -
\bPi \bu\|_{\partial \Omega}) \\ 
\leq C \sigma^{-1} h^{3} \||\nabla
\bbeta| \nabla \bu\|_{H^1(\Omega)}^2 + \sigma\|\bu -
\bPi \bu\|^2_{\Omega} +\sigma h^2 \|\nabla (\bu -
\bPi \bu)\|^2_{\Omega}+ \frac{\sigma}{4} \| \bldeta\|_{\Omega}^2.
\end{multline*}
For the second inequality we used that by \eqref{trace_H1}, $\|\bu -
\bPi \bu\|_{\partial \Omega} \lesssim h^{-\frac12} \|\bu -
\bPi \bu\|_{\Omega} + h^{\frac12} \|\nabla \bu -
\bPi \bu\|_{\Omega}$ and using also \eqref{inverse} it is
straightforward to prove that
\begin{align*}
h^{\frac12}(\| \bldeta\|_{\partial \Omega} + \|\bu -
\bPi \bu\|_{\partial \Omega} )
&\leq C h^{\frac12} \underbrace{\| \bldeta - \bPi
\bldeta\|_{\partial \Omega} }_{= \|\bu -
\bPi \bu\|_{\partial \Omega}} + C \|
\bPi \bldeta\|_{\Omega} +  \|\bu -
\bPi \bu\|_{\Omega} + h  \|\nabla (\bu -
\bPi \bu)\|_{\Omega} 
\\
&\leq C\|\bu -
\bPi \bu\|_{\Omega} + C h  \|\nabla (\bu -
\bPi \bu)\|_{\Omega}  +  C \|
\bldeta\|_{\Omega} .
\end{align*}
It only remains to bound the convective term. To this end we use the
regularized error defined by \eqref{eq:E1_Omega}-\eqref{eq:E3_Omega}
and write
\begin{align}\label{eq:conv_pert}
(\bbeta \cdot \nabla\bldeta, \bu -
\bPi \bu)_\Omega &= (\bbeta \cdot \nabla\bldeta, \nabla \times \bE)_\Omega
\\ \nonumber
&= (\nabla \times (\bbeta \cdot \nabla \bldeta), \bE)_{\mathcal{T}}+ \frac12
\sum_T \int_{\partial T}  (\jump{\bbeta \cdot \nabla \bu_h}\times \bn) \cdot \bE \times
\bn ~\mbox{d}s.
\end{align}
Using now  that the finite element functions are affine per element
and Lemma \ref{derivative-identity} we have
\begin{equation}\label{eq:cross_conv}
(\nabla \times (\bbeta \cdot \nabla \bldeta), \bE)_{\mathcal{T}} = (\bbeta
\cdot \nabla (\nabla \times\bu), \bE)_\Omega+(\nabla \bbeta)^t\times \nabla \bldeta, \bE)_\Omega.
\end{equation}
For the first term of the right hand side we see that
\begin{equation}\label{eq:conv1}
(\bbeta
\cdot \nabla (\nabla \times\bu), \bE)_\Omega \leq \|(\bbeta
\cdot \nabla (\nabla \times\bu)\|_\Omega \|\bE\|_\Omega.
\end{equation}
The second term of the right hand side of \eqref{eq:cross_conv} is bounded in the following way, using the
Cauchy-Schwarz inequality, Young's inequality and Lemma \ref{lem:weak_approx}
\begin{equation}\label{eq:conv2}
((\nabla \bbeta)^t\times \nabla \bldeta, \bE)_\Omega \leq \|\nabla
\bbeta\|^{\frac12}_{\infty} \|\delta |\nabla
\bbeta|_F^{\frac12} \nabla \bldeta\|_\Omega \|\delta^{-1} \bE\|_\Omega
\leq \frac14 |\bldeta|_s^2+ \|\nabla \bbeta\|_{\infty}\|\delta^{-1} \bE\|_\Omega^2
\end{equation}
Finally for the last term of the right hand side of
\eqref{eq:conv_pert} we see that 
\begin{align}\label{eq:conv3}
\sum_T \int_{\partial T} (\jump{\bbeta \cdot \nabla \bu_h\times \bn}) \cdot \bE \times
\bn ~\mbox{d}s &\leq |\bldeta|_s
\gamma^{-\frac12}\|\bbeta\|_{\infty}^{\frac12}h^{-1} \|
\bE\|_{\mathcal{F}} 
\\ \nonumber
&\leq
\frac14 |\bldeta|_s^2 + \gamma^{-1}\|\bbeta\|_{\infty}h^{-3} \|
h^{\frac12} \bE\|_{\mathcal{F}}^2.
\end{align}
Applying the bounds of the equations \eqref{eq:reac_term}-\eqref{eq:conv3} to the
right hand side of \eqref{eq:lin_pert} leads to
\begin{align}\label{eq:lin_approx_bound}
\frac12 \|\sigma^{\frac12} \bldeta\|^2_\Omega + \frac12 |\bldeta|_s^2
&\lesssim \|\sigma^{\frac12}(\bu -
\bPi \bu)\|^2_\Omega +  |\bu - \bPi
\bu|_s^2
\\ \nonumber
&+ C \sigma^{-1} h^{3} \||\nabla
\bbeta| \nabla \bu\|_{H^1(\Omega)}^2+\sigma h^2 \|\nabla (\bu -
\bPi \bu)\|^2_{\Omega}  
\\ \nonumber
&+\|(\bbeta
\cdot \nabla (\nabla \times\bu)\|_\Omega \|\bE\|_\Omega+ \|\nabla \bbeta\|_{\infty}\|\delta^{-1} \bE\|_\Omega^2+\gamma^{-1}\|\bbeta\|_{\infty}h^{-3} \|
h^{\frac12} \bE\|_{\mathcal{F}}^2.
\end{align}
We finish the proof by applying \eqref{eq:L2proj2}, Lemma
\ref{lem:stab_approx} and Lemma \ref{lem:weak_approx} in the right
hand side of \eqref{eq:lin_approx_bound} and the assumption on $\bbeta$.
\end{proof}
\begin{remark}
In the case of the linear model problem the addition of the
linearized Smagorinsky term is not strictly necessary to obtain the
order $h^{\frac32}$ in the estimate of Proposition
\ref{prop:lin_error_est}. Indeed the bound of equation
\eqref{eq:conv2} can be modified as follows to obtain the convergence
with $\delta=0$.
\begin{equation}\label{eq:conv2_mod}
((\nabla \bbeta)^t\times \nabla \bldeta, \bE)_\Omega \leq \|\nabla
\bbeta\|_\infty \|h \nabla \bldeta\|_\Omega \|h^{-1} \bE\|_\Omega.
\end{equation}
Adding and subtracting $\pi_h \bldeta$, using the
triangle inequality and the inverse inequality \eqref{inverse} gives
\[
\|h \nabla \bldeta\|_\Omega \leq \|h \nabla ( \bldeta - \pi_h
\bldeta)\|_\Omega+ \|\pi_h \bldeta\|_\Omega.
\]
Then by approximation and stability of the $L^2$-projection,
\[
\|h \nabla \bldeta\|_\Omega \leq C h^2 |\bu|_{H^2(\Omega)}+ \|\bldeta\|_\Omega.
\]
Using this in \eqref{eq:conv2_mod} followed by an arithmetic-geometric
inequality leads to
\[
((\nabla \bbeta)^t\times \nabla \bldeta, \bE)_\Omega \leq  \frac14 \sigma \|\bldeta\|_\Omega^2 + C h^4
\|\nabla
\bbeta\|_\infty |\bu|_{H^2(\Omega)}^2 + \|\nabla
\bbeta\|_\infty  (C + \sigma^{-1}\|\nabla
\bbeta\|_\infty) \|h^{-1}\bE\|_\Omega^2. 
\]
The first term in the right hand side can now be absorbed by the
zero'th order term in the left hand side and the remaining terms are
of order $O(h^4)$, by Lemma \ref{lem:weak_approx}, which shows that they are sufficiently small by a
margin of $O(h)$. In the nonlinear case below, $\bbeta$ will be
replaced by the discrete solution and it appears no longer to be
possible to balance the estimate to optimal order without the
nonlinear stabilization. So independent of the improved stability
through scale separation, the nonlinear stabilization is of interest
for the finite element discretization of the Navier-Stokes' equations.
\end{remark}
\section{Stabilized finite element method for the incompressible Navier-Stokes' equations}
The finite element space semi-discretization
for the approximation of \eqref{eq:NS_smag} reads: find $
\bhu_h \in \bV_{0h}$, with $\bhu_h(0) = \bPi \bu(0)$ such that
\begin{align}\label{eq:FEM}
(\partial_t \bhu_h, \bv_h)_\Omega &+ ((\bhu_h \cdot \nabla) \bhu_h,
\bv_h)_\Omega + ((\hat \nu(\bhu_h)  + \mu) \nabla \bhu_h,\nabla
\bv_h)_\Omega 
\\ \nonumber
&\qquad \qquad 
+ bc(\bhu_h, \bv_h)+ s(\bhu_h,\bv_h)
= ({\bf f}, \bv_h)_\Omega \qquad 
 \forall  \bv_h \in \bV_{0h}.
\end{align}
Here the term $bc$ is a consistency term added due to the fact that
only the normal component of the velocity is set to zero in $\bV_{0h}
$ defined by
\[
bc(\bhu_h, \bv_h):= -(\mu \nabla \bhu_h \bn ,
t \bv_h )_{\partial \Omega}-(\mu \nabla \bv_h \bn ,
t \bhu_h )_{\partial \Omega},
\]
where $t = I - \bn \otimes \bn$ is the projection onto the tangential
plane of the boundary of $\Omega$. The stabilization term $s$ is defined by
\begin{equation*}
 s(\bhu_h,\bv_h):= s_0(\bhu_h,\bv_h) + s_1(\bhu_h,\bv_h)
\end{equation*}
with
\[
s_0(\bhu_h,\bv_h) := \gamma_0  ( h^{2} (|\bhu_h|+U)^{-1}
 \jump{(\bhu_h\cdot \nabla )\bhu_h\times \bn} ,
 \jump{(\bhu_h\cdot \nabla) \bv_h \times \bn})_{\mathcal{F}}  
\]
and
\[
s_1(\bhu_h,\bv_h):= \gamma_1 ( \max(\mu h^{-1},U) t \bhu_h ,
t \bv_h)_{\partial \Omega}.
\]
The
parameters $\gamma_i$, $i=0,1$ are dimensionless, positive, real numbers and $U = O(1)$
is some characteristic velocity of the flow. 

 The following Lemma is useful for
the analysis in the presence of $bc$. We give the proof of this
results in appendix.
\begin{lemma}\label{lem:bc_stab}
For all $\bv_h \in V_{0h}$ there holds
\[
bc(\bv_h, \bv_h) \leq C  \gamma_1^{-1} \|\mu^{\frac12} \nabla
\bv_h\|^2_{\Omega}+ \frac14 \gamma_1\max(\mu h^{-1},U) \|t \bv_h\|^2_{\partial
  \Omega}.
\]
For all $\bv \in \bH^2(\Omega) + \bV_{0h}$ there holds
\[
bc(\bv, \bv_h) \leq  C\gamma_1^{-1} \|\mu^{\frac12} \nabla
\bv\|^2_{\Omega} + C h^2 \mu \|\bv
\|^2_{\bH^2(\mathcal{T})}  + \frac14 \gamma_1 \max(\mu h^{-1},U) \|t \bv_h\|^2_{\partial
  \Omega}.
\]
\end{lemma}
The formulation
\eqref{eq:FEM} corresponds to a dynamical system. It admits a unique
solution as shown in the following propositions.
\begin{proposition}\label{prop:apriori}
For $\gamma_1$ large enough, a solution $\bhu_h $ to the system \eqref{eq:FEM}  satisfies the following
stability estimate for all $T>0$,
\begin{equation}
\sup_{t \in I} \|\bhu_h(t)\|_\Omega + \|\mu^{\frac12} \nabla
\bhu_h\|_{Q}+ \|\delta^{\frac23}\nabla
\bhu_h\|^{\frac32}_{L^3(Q)}+\left(\int_I s(\bhu_h,\bhu_h) ~\mbox{d}t\right)^{\frac12}
\lesssim \int_I \|{\bf f }\|_\Omega ~\mbox{d}t+\|\bhu_h(0) \|_\Omega,
\end{equation}
where the constant $C$ is independent of $T$.
\end{proposition}
\begin{proof}
Testing \eqref{eq:FEM} with $\bv_h = \bhu_h$ yields, for all $t \in I$,
\begin{equation}\label{eq:stab1}
(\partial_t \bhu_h, \bhu_h)_\Omega + \underbrace{((\bhu_h \cdot
  \nabla) \bhu_h, \bhu_h)_\Omega}_{=0} + ((\hat \nu(\bhu_h)  + \mu) \nabla \bhu_h,\nabla
\bhu_h)_\Omega + bc(\bhu_h,\bhu_h) + s(\bhu_h,\bhu_h)
= ({\bf f}, \bhu_h)_\Omega.
\end{equation}
Noting that the second term is zero by skew-symmetry we have after
integration on $(0,r)$ for all $r \in I$ and using Lemma \ref{lem:bc_stab}
\begin{multline}\label{eq:stab2}
\frac12 \|\bhu_h(r) \|_\Omega^2+\|\mu^{\frac12} \nabla
\bhu_h\|_{\Omega \times (0,r)}^2+\delta^2 \|\nabla
\bhu_h\|^{3}_{\Omega \times (0,r)} + \int_0^r s(\bhu_h,\bhu_h) ~\mbox{d}t
\\
\lesssim\int_0^r ({\bf f}, \bhu_h)_\Omega ~\mbox{d}t + \frac12 \|\bhu_h(0) \|_\Omega^2.
\end{multline}
Taking now the supremum over $r \in I$ in the first term of the left
hand side we see that
\[
\frac12\sup_{r \in I}  \|\bhu_h(r) \|_\Omega^2 \lesssim \sup_{r \in I}
\|\bhu_h(r) \|_\Omega (\int_I \|{\bf f }\|_\Omega ~\mbox{d}t + \frac12 \|\bhu_h(0) \|_\Omega)
\]
and therefore
\[
\sup_{r \in I}  \|\bhu_h(r) \|_\Omega \lesssim \int_I \|{\bf f} \|_\Omega ~\mbox{d}t+\|\bhu_h(0) \|_\Omega.
\]
Finally we see that
\begin{multline*}
\|\mu^{\frac12} \nabla
\bhu_h\|_{\Omega \times (0,r)}^2+ \delta^2\|\nabla
\bhu_h\|^{3}_{\Omega \times (0,r)}+\int_0^r s(\bhu_h,\bhu_h) ~\mbox{d}t 
\lesssim  \int_0^r ({\bf f}, \bhu_h)_\Omega ~\mbox{d}t + \frac12 \|\bhu_h(0) \|_\Omega^2 \\\leq \sup_{r \in I}
\|\bhu_h(r) \|_\Omega \left(\int_I \|{\bf f }\|_\Omega ~\mbox{d}t + \frac12
\|\bhu_h(0) \|_\Omega\right)
\lesssim \left(\int_I \|{\bf f }\|_\Omega ~\mbox{d}t + \|\bhu_h(0) \|_\Omega\right)^2
\end{multline*}
which finishes the proof.
\end{proof}

\begin{proposition}
The system \eqref{eq:FEM}  admits a
unique solution $\bhu_h \in \bV_{0h}$ on the interval $(0,T]$, $T>0$.
\end{proposition}
\begin{proof}
Let $N = \mbox{dim} ~\bV_{0h}$, $\bU: \mathbb{R} \mapsto
\mathbb{R}^N$, $\bF: \mathbb{R}^N \mapsto
\mathbb{R}^N$ and let $\{\bvarphi_i\}$ denotes a basis of $\bV^0_{0h}$. The system \eqref{eq:FEM} is equivalent to a dynamical
system 
\[
M\partial_t \bU(t) + \bF(\bU) = 0,
\]
with $\bU(0) = \bU_0 \in \mathbb{R}^N$
\[
(M\partial_t \bU(t))_i =  (\partial_t \bhu_h, \bvarphi_i)_\Omega,\, i=1,\hdots,N
\]
and
\[
(\bF(\bU))_i :=((\bhu_h \cdot \nabla) \bhu_h, \bvarphi_i)_\Omega + (\mu \nabla \bhu_h,\nabla \bvarphi_i)_\Omega + s(\bhu_h,\bvarphi_i)
- ({\bf f}, \bvarphi_i)_\Omega,\, i=1,\hdots,N.
\]
The matrix $M$ corresponds to
the mass matrix of the velocity finite element basis, defined
blockwise by $M_{ij}:= (\bvarphi_j,\bvarphi_i)_\Omega$. 

By inspection the function
$\bF(\bU)$ is locally Lipschitz. Indeed, for all $\bU, \bY \in \mathbb{R}^N$ there holds
\begin{equation}\label{eq:Lip}
|\bF(\bU) - \bF(\bY)|_{\mathbb{R}^N} \leq L(\bU, \bY) |\bU - \bY|_{\mathbb{R}^N}.
\end{equation}
To verify this we only need to consider the nonlinear terms. Let $\bu_h$ and $\by_h$ denote the functions in
$\bV_{0h}$ associated to the vectors of unknowns $\bU$ and $\bV$ and let $\bw_h = \bu_h -
\by_h$. Then we have
\begin{multline}
(\bF(\bU))_i -  (\bF(\bY))_i = ((\bu_h \cdot \nabla) \bu_h,\bvarphi_i)_\Omega+
(\hat \nu(\bu_h) \nabla \bu_h,\nabla \bvarphi_i)_\Omega + s(\bu_h,\bvarphi_i) \\
-
((\by_h \cdot \nabla) \by_h, \bvarphi_i)_\Omega -  (\hat \nu(\by_h) \nabla \by_h,\nabla \bvarphi_i)_\Omega-  s(\by_h,\bvarphi_i).
\end{multline}
The convective term and the face oriented contribution to the
stabilization both are $[C^1(\mathbb{R}^N)]^N$ and the Lipschitz
continuity follows from the mean value theorem. For the Smagorinsky
term on the other hand it is immediate by the inequality \eqref{eq:continuous}.
It follows that for every $h>0$ there
 exists a $T>0$ such that \eqref{eq:FEM} admits a unique solution on
 $(0,T]$. To extend this to arbitrary time intervals we need a bound on
 $\|\bu_h\|_{L^\infty(\Omega)}$. 
 We apply an inverse inequality
 $\|\bu_h(\cdot,T)\|_{L^\infty(\Omega)} \leq C h^{-\frac{d}{2}}
 \|\bu_h(\cdot,T)\|_\Omega$. Recalling Lemma \ref{prop:apriori} it follows that for fixed $h$ the
 solution is globally bounded. 
This proves the claim.
\end{proof}
\begin{remark}
A consequence of the error analysis below is that the $L^\infty$-bound
on the discrete solution actually holds independently of $h$ provided the
exact solution is smooth. For details see Corollary \ref{cor:inftybound}
\end{remark}
\subsection{Error analysis with exponential growth moderated
  through scale separation} 
We will now prove an error
estimate for the approximation of the regularized equation. We are
interested in underresolved flow so in what follows we assume that
$\mu \leq U h$. Observe that the arguments are valid also for $\mu=0$,
 but in this case the no-slip condition must be relaxed on the
 continuous level. Moreover the same analysis may be used to prove an
 optimal estimate of $O(h)$ in the $H^1$-norm if $\mu \ge U h$.
To moderate the exponential growth
we use a scale separation argument similar to that of section \ref{sec:scale_pert}.

\subsubsection{Scale separation for the finite dimensional
  approximation}
When a finite dimensional space is used for the Smagorinsky-Navier-Stokes' model we
can introduce a scale separation argument similar to \eqref{eq:scale_separation}, but
here the nonlinear feedback in the second equation of \eqref{eq:scale_separation} takes place through the computational
error. As the computational error grows, the exponential growth of the
error is moderated. More precisely define $\bar \bu$ and $\bu'$ by
\eqref{eq:scale_separation} with $\bldeta$ defined by $\bldeta := \bu
- \bhu_h$. We also note that then the result of Lemma \ref{lem:conv_term} holds with $\bldeta$
redefined to be the approximation error. Below we will refer to these
results assuming that $\bldeta$ is redefined as above.
\subsubsection{Error estimates}
 First we prove a Lemma that is
needed to estimate the consistency error of the stabilization term to
the right order.
\begin{lemma}\label{lem:nonline_stab}
For $\bv_h \in \bV_{0h}$, and $f \in L^3(I)$
there holds, for all $ \bv \in \bV$  and for all $\epsilon >0$,
\begin{align*}
&\int_I \|\bv_h(t)\|_{L^\infty(\Omega)} |f(t)|^2 ~\mbox{d}t 
\\
&\qquad \leq \epsilon
\delta^2 \|\nabla(\bv - \bv_h)\|^3_{L^3(Q)} + C\frac{1}{\epsilon} 
\delta^{-1} h^{-\frac{1}{2}}\int_I   |f(t)|^3 ~\mbox{d}t
 +C \int_I \|\bv\|_{W^{2,3}(\Omega)}|f(t)|^2 ~\mbox{d}t.
\end{align*}
\end{lemma}
\begin{proof}
Let $\bw:=\bv- \bv_h$ and $\bw_\pi:=\bv- \bpi_h \bv$.
First we add and subtract $\bpi_h\bv$  to obtain
\[
\|\bv_h(t)\|_{L^\infty(\Omega)}  \leq \|\bv_h(t)-\bpi_h\bv\|_{L^\infty(\Omega)}+\|\bpi_h\bv\|_{L^\infty(\Omega)}.
\]
Then using an inverse inequality \eqref{eq:Lpinv} followed by Morrey's inequality
\cite[Theorem B.42]{EG04} and Lemma \ref{lem:PF} we
see that for all $p<\infty$
\[
\|\bv_h(t)-\bpi_h\bv\|_{L^\infty(\Omega)} \leq C_p h^{-\frac{d}{p}}
(\|\nabla \bw(t) \|_{L^3(\Omega)}+\|\bw_{\pi}(t) \|_{W^{1,3}(\Omega)}) .
\]
Using also the stability of the $L^2$-projection on quasi-uniform
meshes we have
\[
\|\bv_h(t)\|_{L^\infty(\Omega)} \leq C_p h^{-\frac{d}{p}}(
\|\nabla \bw(t) \|_{L^3(\Omega)}+\|\bw_{\pi}(t) \|_{W^{1,3}(\Omega)})  +
C \|\bv\|_{L^\infty(\Omega)}.
\]
It follows that
\begin{align*}
\int_I \|\bv_h(t)\|_{L^\infty(\Omega)} |f(t)|^2 ~\mbox{d}t 
&\leq \int_I (C_p h^{-\frac{d}{p}}(
\|\nabla \bw(t) \|_{L^3(\Omega)}+\|\bw_{\pi}(t) \|_{W^{1,3}(\Omega)})  +
C \|\bv\|_{L^\infty(\Omega)}) |f(t)|^2 ~\mbox{d}t 
\\
&\leq \epsilon \delta^2 \|\nabla \bw(t) \|^3_{L^3(Q)} +
\frac{1}{\epsilon} C _p^{\frac32} \delta^{-1} h^{-\frac{3 d}{2
    p}}\int_I |f(t)|^3 ~\mbox{d}t    
    \\
    &\qquad 
+\int_I (C_p h^{-\frac{d}{p}}\|\bw_{\pi}(t) \|_{W^{1,3}(\Omega)} + C \|\bv\|_{L^\infty(\Omega)}) |f(t)|^2 ~\mbox{d}t.
\end{align*}
The claim now follows by taking $p=9$ and observing that $\|\bw_{\pi}(t)
\|_{W^{1,3}(\Omega)} \leq C h |\bv|_{W^{2,3}(\Omega)}$ and
$\|\bv\|_{L^\infty(\Omega)} \leq  \|\bv\|_{W^{2,3}(\Omega)}$.
\end{proof}
One application of the previous lemma is the following consistency
bound for the stabilizing term
\begin{lemma}\label{lem:stab_consist}
The stabilizing term satisfies the bound
\[
\int_I s(\bPi \bu,\bPi \bu) ~\mbox{d}t \leq \epsilon
\delta^2 \|\nabla(\bu - \bhu_h)\|^3_{L^3(Q)}   + C \max(\mu h^{-1}, U)
 h^3 \|\bu\|_{L^2(I;H^2(\Omega))}^2 + C h^3  \|\bu\|_{L^3(I;W^{2,3}(\Omega))}^3.
\]
\end{lemma}
\begin{proof}
By the definition of the stabilization term we have the bound
\[
 s(\bPi \bu,\bPi \bu)\leq \gamma_0 \|
 \bhu_h\|_{L^{\infty}(\Omega)} \|h \jump{\nabla
   \bPi \bu}\|_{\mathcal{F}}^2 + \gamma_1 \max(\mu h^{-1}, U)\|t \bPi \bu\|^2_{\partial \Omega}.
\]
Then by Lemma \ref{lem:nonline_stab},
$s$ admits the bound
\begin{multline}\label{eq:stab_approx_nonlin}
\int_I  s(\bPi \bu,\bPi \bu) ~\mbox{d}t \leq \epsilon
\delta^2 \|\nabla (\bu - \bhu_h)\|^3_{L^3(Q)} + C
\delta^{-1} h^{-\frac{1}{2}}\int_I   \|h \jump{\nabla
   \bPi \bu}\|_{\mathcal{F}}^3 ~\mbox{d}t\\
 +C \int_I \|\bu\|_{W^{2,3}(\Omega)}\|h \jump{\nabla
   \bPi \bu}\|_{\mathcal{F}}^2 ~\mbox{d}t + \gamma_1 \max(\mu h^{-1},
 U) \int_I \|t \bPi \bu\|^2_{\partial \Omega} ~\mbox{d}t.
\end{multline}
Using \eqref{eq:jump_bound}, and $|\bu|_{H^2(\Omega)} \lesssim \|\bu\|_{W^{2,3}(\Omega)}$ we have
\[
C
\delta^{-1} h^{-\frac{1}{2}}\int_I   \|h \jump{\nabla
   \bPi \bu}\|_{\mathcal{F}}^3 ~\mbox{d}t
 +C \int_I \|\bu\|_{W^{2,3}(\Omega)}\|h \jump{\nabla
   \bPi \bu}\|_{\mathcal{F}}^2 \leq C h^{-\frac32} h^{\frac92}
 |\bu|_{L^2(I;H^2(\Omega))}^3 + C h^3  \|\bu\|_{L^3(I;W^{2,3}(\Omega))}^3 
\]
The no-slip condition satisfied by $\bu$, the trace inequality \eqref{trace_H1} and the
approximation \eqref{eq:L2proj2} may be applied to control the second term,
\begin{align*}
\gamma_1 \max(\mu h^{-1}, U)\|t \bPi \bu\|^2_{\partial \Omega} 
&\leq C
\max(\mu h^{-1}, U) (h^{-1} \|\bu - \bPi \bu\|^2_{ \Omega}  +
h  \|\nabla (\bu - \bPi \bu)\|^2_{ \Omega} ) 
\\
&\leq  C \max(\mu h^{-1}, U)
 h^3 \|\bu\|_{L^2(I;H^2(\Omega))}^2.
\end{align*}
The claim follows by collecting the above bounds.
\end{proof}
\begin{theorem}\label{thm:error_bound}
Let $\bu \in 
\bV$ and that $\bu$ in addition has sufficient regularity so that
$\mathcal{C}(\bu) < \infty$ if we define (assuming $U\leq\|\bu\|_{L^3(I;W^{2,3}(\Omega))}$),
\begin{multline}\label{eq:reg_const}
 \mathcal{C}(\bu)^2:= 
h (\tau_L \|\nabla \times \bu\|_{L^2(I;L^\infty(\Omega))}
+1)|\bu|^2_{L^\infty(I;H^2(\Omega))} \\
+
\|\bu\|^3_{L^3(I;W^{2,3}(\Omega))}+\|\bu\|^4_{L^4(I;W^{2,3}(\Omega))} + h
\|\nabla \cdot |\nabla \bu|_F \nabla \bu\|^2_Q
\end{multline}
Assume that the parameter $\delta = O(h)$ and $\mu \leq U h$, and that
the hypothesis of
Theorem \ref{prop:apriori} hold. Then there holds
\[
\sup_{t \in I} \|(\bu - \bhu_h)(t)\|_\Omega + \tnorm{\bu - \bhu_h}  \lesssim \mathcal{C}(\bu)
e^{18(T/\tau_L)} h^{\frac32},
\]
where 
\[
\tnorm{\bu - \bhu_h} := \|\mu^{\frac12}
\nabla (\bu - \bhu_h)(t)\|_Q + \tau_L^{\frac12} h \|\nabla (\bu -
\bhu_h)\|^{3/2}_{L^3(Q)} + \tau_L^{\frac12} \left(\int_I s(\bhu_h,\bhu_h)~\mbox{d}t
\right)^{\frac12}.
\]
The coefficient $\tau_{L}$ in the
exponential is defined by \eqref{eq:char_time}
where $\bar \bu$ this time is the large scale component of the solution defined
by \eqref{eq:scale_separation}, with $\bldeta = \bu - \bhu_h$. The
hidden constant has at most polynomial growth in $T$.
\end{theorem}
\begin{proof}
The proof follows the ideas of the result for the linear model problem
Proposition \ref{prop:lin_error_est}, but the nonlinearity of the
equation and the stabilization adds a layer of technicality. We
proceed in 4 steps.
\paragraph{\bf{Step 1. Perturbation equation}}
Let $\bldeta = \bu - \bhu_h$, $\bldeta_\Pi = \bu - \bPi \bu$ and
$\bldeta_\pi = \bu - \bpi \bu$. Similarly as in the perturbation argument
for the Smagorinsky model we have, since $\nabla \cdot \bhu_h = 0$ and
$\bhu_h\cdot \bn\vert_{\partial \Omega} = 0$,
\begin{align*}
&\frac12 \frac{d}{dt} \|\bldeta\|_\Omega^2 + ((\bldeta \cdot \nabla)
\bu,\bldeta)_\Omega +  \mu\|\nabla \bldeta\|^2_{\Omega}+\frac14 \delta^2
\|\nabla \bldeta\|_{L^3(\Omega)}^3 + s(\bhu_h,\bhu_h) 
\\
&\qquad 
\leq (\partial_t \bldeta, \bldeta)_\Omega + ((\bldeta \cdot \nabla)
\bu,\bldeta)_\Omega+ ((\bhu_h \cdot \nabla) \bldeta,\bldeta)_\Omega + (\mu \nabla \bldeta, \nabla \bldeta)_\Omega
\\
&\qquad \qquad + (\hat \nu(\bu) \nabla \bu - \hat \nu(\bhu_h)
\nabla \bhu_h, \nabla \bldeta) + s(\bhu_h,\bhu_h).
\end{align*}
Using the definition of the finite element method \eqref{eq:FEM} we have
the consistency relation
\begin{align*}
& (\partial_t \bldeta, \bldeta)_\Omega + ((\bldeta \cdot \nabla)
\bu,\bldeta)_\Omega + ((\bhu_h \cdot \nabla) \bldeta,\bldeta)_\Omega + (\mu \nabla \bldeta, \nabla \bldeta)_\Omega
+ (\hat \nu(\bu) \nabla \bu - \hat \nu(\bhu_h)
\nabla \bhu_h, \nabla \bldeta) + s(\bhu_h,\bhu_h)
\\
&\qquad 
=  (\partial_t \bldeta, \bldeta_\Pi)_\Omega + ((\bldeta \cdot \nabla)
\bu,\bldeta_\Pi)_\Omega + ((\bhu_h \cdot \nabla)
\bldeta,\bldeta_\Pi)_\Omega + (\mu \nabla \bldeta, \nabla
\bldeta_\Pi)_\Omega + bc(\bldeta,\bPi \bldeta)
\\
&\qquad \qquad 
+(\hat \nu(\bu) \nabla \bu, \nabla
\bPi \bldeta)_\Omega+ (\hat \nu(\bu) \nabla \bu - \hat \nu(\bhu_h) \nabla \bhu_h, \nabla
\bldeta_\Pi)_\Omega + s(\bhu_h,\bPi \bu)
.
\end{align*}
Integrating in time over the interval $I$ yields
\begin{align}\label{eq:1st_bound}
&\frac12\|\bldeta(T)\|_\Omega^2 + ((\bldeta \cdot \nabla)
\bu,\bldeta)_Q+ \frac14 \|\mu^{\frac12} \nabla \bldeta \|_Q^2 +\frac14 \delta^2
\|\nabla \bldeta(t)\|_{L^3(Q)}^3 + \frac12\int_I s(\bhu_h,\bhu_h) ~\mbox{d}t
\\ \nonumber
&\qquad \leq (\partial_t \bldeta, \bldeta_\Pi)_Q + ((\bldeta \cdot \nabla)
\bu,\bldeta_\Pi)_Q 
+ ((\bhu_h \cdot \nabla) \bldeta,\bldeta_\Pi)_Q + (\mu \nabla \bldeta, \nabla \bldeta_\Pi)_Q
\\ \nonumber
&\qquad \qquad + (\hat \nu(\bu) \nabla \bu - \hat \nu(\bhu_h)
\nabla \bhu_h, \nabla \bldeta_\Pi)_Q - (\hat \nu(\bu) \nabla \bu, \nabla
\Pi \bldeta)_Q  
\\ \nonumber
&\qquad \qquad 
+ \frac12\|\bldeta_{\Pi}(0)\|_\Omega^2 +  \int_I (bc(\bldeta,\bPi \bldeta)+ s(\bhu_h,\bPi \bu)) ~\mbox{d}t.
\end{align}
\paragraph{\bf{Step 2. Continuity and approximation properties of the right hand side of
  the perturbation equation}}
We now bound the terms in the right hand side. There are two different
main difficulties involving nonlinearities. The material derivative and the viscous term on
the one hand and the nonlinear stabilization on the other.

\noindent {\bf 2.1. Material derivative and viscous term.}
The following term will be bounded in this paragraph.
\[
 (\partial_t \bldeta, \bldeta_\Pi)_Q + ((\bldeta \cdot \nabla)
\bu,\bldeta_\Pi)_Q 
+ ((\bhu_h \cdot \nabla) \bldeta,\bldeta_\Pi)_Q + (\mu \nabla \bldeta, \nabla \bldeta_\Pi)_Q.
\]
First for the time
derivative observe that using the orthogonality of the approximation
error $\bldeta_\Pi$,
\[
(\partial_t \bldeta, \bldeta_\Pi)_Q = (\partial_t \bldeta_\Pi,
\bldeta_\Pi)_Q \leq C h^4 |\bu|_{L^\infty(I;H^2(\Omega))}^2
\]
and for the linear viscous term 
\[
 (\mu \nabla \bldeta, \nabla
\bldeta_\Pi)_Q \leq \frac18 \|\mu^{\frac12} \nabla \bldeta\|_Q^2 + 2
\|\mu^{\frac12} \nabla \bldeta_\Pi\|^2_Q \leq \frac{1}{16} \|\mu^{\frac12}
\nabla \bldeta\|_Q^2  + C \mu h^2 |\bu|_{L^2(I;H^2(\Omega))}^2.
\]
For the convective part we write
\begin{equation}\label{eq:convect_NS}
((\bldeta \cdot \nabla)
\bu,\bldeta_\Pi)_Q 
+ ((\bhu_h \cdot \nabla) \bldeta,\bldeta_\Pi)_Q = I + II.
\end{equation}
For the first term $I$ we use similar arguments as those for Lemma
\ref{lem:conv_term} to get the bound
\[
I \leq \frac12 \|(\tau_L^{-1} + |\nabla \bar \bu|_F)^{\frac12}
\bldeta\|_Q^2 
+ \frac{1}{32} \delta^2 \|\nabla \bldeta \|_{L^3(Q)}^3 +  C\|\mu^{\frac12} \nabla \bldeta_\Pi \|_{L^2(Q)}^2+ C \delta^2
\|\nabla \bldeta_\Pi\|_{L^3(Q)}^3.
\]
In particular note that
\begin{multline}
(|\bu'|,(\bldeta \cdot \nabla)\bldeta_\Pi)_Q \leq \|\tau_L^{-\frac12}
\bldeta\|_Q \||(\mu+\nu(\bhu_h))^{\frac12} \nabla \bldeta_\Pi\|_Q\\
\leq \frac12 \|\tau_L^{-\frac12}
\bldeta\|_Q^2 + \frac{1}{32} \delta^2 \|\nabla \bldeta \|_{L^3(Q)}^3 +
C \|\mu^{\frac12}\nabla \bldeta_\Pi\|_{L^2(Q)}^2+ C \delta^2
\|\nabla \bldeta_\Pi\|_{L^3(Q)}^3.
\end{multline}
By Lemma \ref{lem:lptol2} we may write the last term in the right hand side
\[
\delta^2
\|\nabla \bldeta_\Pi\|_{L^3(Q)}^3 \lesssim \delta^2
\|\nabla \bldeta_\pi\|_{L^3(Q)}^3 + \delta^2 h^{-\frac32} \|\nabla \bldeta_\Pi\|_{L^2(Q)}^3.
\]
Since $\delta^2 h^{-\frac32} =
h^{\frac12}$ and $|\bu|_{H^2(\Omega)} \leq |\bu|_{W^{2,3}(\Omega)}$ we
have using approximation that, 
\[
I \leq \frac12 \|(\tau_L^{-1} + |\nabla \bar \bu|_F)^{\frac12}
\bldeta\|_Q^2 
+ \frac{1}{32} \delta^2 \|\nabla \bldeta \|_{L^3(Q)}^3 + \mu h^2
|\bu|_{L^2(I;H^2(\Omega))}^2 +  h^{\frac72} |\bu|_{L^3(I;W^{2,3}(\Omega))}^3.
\]
For the second term of the right hand side of \eqref{eq:convect_NS} we first use that there exists $\bE$ according to
\eqref{eq:E1_Omega}-\eqref{eq:E3_Omega} to write using partial integration
\begin{align*}
II &= ((\bhu_h \cdot \nabla) \bldeta,\nabla \times \bE)_Q = 
\int_I \underbrace{
(\jump{(\bhu_h \cdot \nabla) \bldeta \times \bn}, \bE \times
\bn)_{\mathcal{F}}}_{II_a}~\mbox{d}t  
\\
&\qquad + \int_I \underbrace{((\bhu_h \cdot \nabla) (\nabla \times \bldeta) ,
\bE)_{\mathcal{T}} }_{II_b} ~\mbox{d}t +\underbrace{\int_I ((\nabla \bhu_h)^t\times \nabla \bldeta) ,
\bE)_{\mathcal{T}} ~\mbox{d}t }_{II_c}.
\end{align*}
For the first term $II_a$ we use Cauchy-Schwarz inequality followed by the
arithmetic-geometric inequality
\begin{align*}
II_a &\leq \left(\int_I s_0(\bhu_h,\bhu_h) ~\mbox{d}t \right)^{\frac12}
\left(\gamma_0^{-1}\int_I
(\|\bhu_h\|_{L^{\infty}(\Omega)} + U)\|h^{-1} \bE
\|_{\mathcal{F}}^2 ~\mbox{d}t \right)^{\frac12} \\
\\
&\leq
\frac18\int_I s_0(\bhu_h,\bhu_h) ~\mbox{d}t + 2 \gamma_0^{-1}\int_I
(\|\bhu_h\|_{L^{\infty}(\Omega)}+U) \|h^{-1} \bE
\|_{\mathcal{F}}^2 ~\mbox{d}t .
\end{align*}
Applying now Lemma \ref{lem:nonline_stab} to the right hand side we
have for $\epsilon>0$,
\begin{align*}
II_a & \leq \frac18 \int_I s_0(\bhu_h,\bhu_h) ~\mbox{d}t +C\epsilon
\delta^2 \|\bldeta\|^3_{L^3(I;W^{1,3}(\Omega))} 
\\
&\qquad 
+ \frac{C}{\epsilon} 
\delta^{-1} h^{-\frac{1}{2}}\int_I   \|h^{-1} \bE
\|_{\mathcal{F}}^3 ~\mbox{d}t
 +C \int_I ( \|\bu\|_{W^{2,3}(\Omega)}+U)\|h^{-1} \bE
\|_{\mathcal{F}}^2 ~\mbox{d}t 
\\
&\leq \frac18 \int_I s_0(\bhu_h,\bhu_h) ~\mbox{d}t 
\\
&\qquad 
+C\epsilon \delta^2 \|\bldeta\|^3_{L^3(I;W^{1,3}(\Omega))} +\frac{C}{\epsilon} \underbrace{ (h^{-9/2}
h^{15/2}}_{ = h^3} + h^3)(U^3 +\|\bu\|^3_{L^3(I;W^{2,3}(\Omega))}).
\end{align*}
Here we used the bounds
\[
\delta^{-1} h^{-\frac{1}{2}}\int_I   \|h^{-1} \bE
\|_{\mathcal{F}}^3 ~\mbox{d}t \lesssim h^{-9/2}
h^{15/2} \|\bu\|^3_{L^2(I;H^{2}(\Omega))} \lesssim h^3 \|\bu\|^3_{L^3(I;W^{2,3}(\Omega))}
\]
and
\begin{align*}
\int_I ( \|\bu\|_{W^{2,3}(\Omega)}+U)\|h^{-1} \bE
\|_{\mathcal{F}}^2 ~\mbox{d}t  &\lesssim
(\|\bu\|_{L^3(I;W^{2,3}(\Omega)}) + T^{1/3} U) h^3
\|\bu\|^2_{L^{3}(I;H^2(\Omega)} 
\\
&\lesssim h^3(T U^3 +\|\bu\|^3_{L^3(I;W^{2,3}(\Omega))}).
\end{align*}
For the second term $II_b$ 
we use the fact that $\bhu_h$ is piecewise affine, add and subtract $\bu$ and then integrate by parts
\begin{align*}
II_b &= - ((\bu \cdot \nabla) (\nabla \times \bu) ,
\bE)_Q- ((\bldeta \cdot \nabla) (\nabla \times \bu) ,
\bE)_Q 
\\
&\leq \int_I \|\bu\|_{L^\infty(\Omega)}  |
\bu|_{H^2(\Omega)} \|\bE\|_{\Omega} ~\mbox{d}t + \int_I\|\bldeta\|_\Omega  \|\nabla \times
\bu\|_{L^\infty(\Omega)} \|\nabla \bE\|_{\Omega}~\mbox{d}t
\\
&\leq   \frac14 \tau_L^{-1} \|\bldeta\|_Q^2
+  h^3 (|\bu|^3_{L^3(I;H^2(\Omega))} +h \tau_L  \|\nabla \times
\bu\|^2_{L^2(I;L^\infty(\Omega))}  |\bu|^2_{L^{\infty}(I;H^2(\Omega))} )
\end{align*}
Where we used that
\[
\int_I \|\bu\|_{L^\infty(\Omega)}  |
\bu|_{H^2(\Omega)} \|\bE\|_{\Omega} ~\mbox{d}t \lesssim h \|
\bu\|^2_{L^3(I; H^2(\Omega)} h^{-1} \|\bE\|_{L^3(I;L^2(\Omega))}
\lesssim h^3 \|\bu\|^3_{L^3(I;H^2(\Omega))}
\]
and
\begin{align*}
\int_I\|\bldeta\|_\Omega  \|\nabla \times
\bu\|_{L^\infty(\Omega)} \|\nabla \bE\|_{\Omega}~\mbox{d}t
&\lesssim
\tau_L^{-1/2} \|\bldeta\|_Q \tau^{1/2}  \|\nabla \times
\bu\|_{L^2(I;L^\infty(\Omega))} h^2 |\bu|_{L^{\infty}(I;H^2(\Omega))}
\\
&\leq 
\frac14 \tau_L^{-1} \|\bldeta\|_Q^2 + C h^4 \tau_L  \|\nabla \times
\bu\|^2_{L^2(I;L^\infty(\Omega))}  |\bu|^2_{L^{\infty}(I;H^2(\Omega))} .
\end{align*}
For $II_c$ we add and subtract $\bu$ and $\pi_h \bldeta$, 
\begin{multline}
II_c = ((\nabla \bhu_h)^t \times \nabla \bldeta,
\bE)_Q = - ((\nabla \bldeta)^t \times \nabla \bldeta ,
\bE)_Q + ((\nabla \bu)^t \times \nabla \bldeta,
\bE)_Q \\
= - ((\nabla \bldeta)^t \times \nabla \bldeta ,
\bE)_Q + ((\nabla \bu)^t \times \nabla (\bldeta - \pi_h \bldeta) ,
\bE)_Q + ((\nabla \bu)^t \times  \nabla \pi_h \bldeta ,
\bE)_Q.
\end{multline}
Observing that $\bldeta - \pi_h \bldeta = \bu - \pi_h \bu$ it follows,
using H\"older's inequality, Sobolev injection and approximation, that
\begin{align*}
&((\nabla \bu)^t \times \nabla (\bldeta - \pi_h \bldeta) ,
\bE)_{\Omega} \leq \|\nabla \bu\|_{L^{4}(\Omega)} \|\nabla (\bu -
\pi_h \bu)\|_{L^{4}(\Omega)} \|\bE\|_\Omega 
\\
&\qquad \qquad \lesssim  \|\nabla
\bu\|_{L^{4}(\Omega)}^2 \|\bE\|_\Omega
\lesssim \|\bu\|^2_{W^{2,3}(\Omega)} h^3 |\bu|_{H^2(\Omega)}\leq h^3 \|\bu\|^3_{W^{2,3}(\Omega)}.
\end{align*}
Once again adding and subtracting $\pi_h \bldeta$ we have
\[
 ((\nabla \bldeta)^t \times \nabla \bldeta ,
\bE)_\Omega \leq C \|\nabla (\bu- \pi_h \bu)\|^2_{L^4(\Omega)}
\|\bE\|_\Omega + C \|\nabla \pi_h \bldeta\|^2_{L^4(\Omega)}\|\bE\|_{\Omega}).
\]
Using an inverse inequality \eqref{eq:Lpinv} and the stability of the projection $\pi_h$ we see that
\[
\|\nabla \pi_h \bldeta\|^2_{L^4(\Omega)} \leq C h^{-\frac16} \|\nabla \bldeta\|^2_{L^3(\Omega)}.
\]
Consequently
\[
\|\nabla \pi_h \bldeta\|^2_{L^4(\Omega)}\|\bE\|_{\Omega} \leq C h^{-\frac16}
\|\nabla \bldeta\|^2_{L^3(\Omega)} \|\bE\|_{\Omega} \leq \frac{1}{32}
\delta^2 \|\nabla \bldeta\|^3_{L^3(\Omega)} + \underbrace{ C h^{-\frac12} \delta^{-4} \|\bE\|_{\Omega}^3}_{*}.
\]
The term marked $(*)$ needs to scale as $O(h^3)$. Since
$\|\bE\|_{\Omega}^3 = O(h^9)$ we see that we need  $h^{-\frac12} \delta^{-4}
\leq h^{-6}$, which is satisfied for $\delta = O(h)$.
We arrive at the following bound for $II_c$,
\begin{equation*}
II_c \leq   \frac{1}{32}
\delta^2 \|\nabla \bldeta\|^3_{L^3(Q)} + C h^3 \|\nabla
\bu\|_{L^3(I;W^{2,3}(\Omega))}^3 
\end{equation*}
where we used the bound
\[
 h^{-\frac12} \delta^{-4} \|\bE\|_{\Omega}^3 \lesssim h^{-9/2} h^9
 |\bu|_{H^2(\Omega)}^3 \lesssim h^{9/2} |\bu|_{W^{2,3}(\Omega)}^3.
\]
Collecting the above bounds for the terms $I$ and $II$ and using that
$\mu \leq U h$, we see that
\begin{multline}\label{eq:convect}
(\partial_t \bldeta, \bldeta_\Pi)_Q   + ((\bldeta \cdot \nabla)
\bu,\bldeta_\Pi)_Q 
+ ((\bhu_h \cdot \nabla) \bldeta,\bldeta_\Pi)_Q +(\mu \nabla \bldeta, \nabla
\Pi \bldeta)_Q \\
\leq   \frac12\int_I \tau_L^{-1}  \|\bldeta\|_\Omega^2 ~\mbox{d}t   + \frac{1}{16} \|\mu^{\frac12}
\nabla \bldeta\|_Q^2+ \frac{1}{16}
\delta^2 \|\nabla \bldeta\|^3_{L^3(Q)} + \frac18 \int_I s_0(\bhu_h,\bhu_h)
~\mbox{d}t 
+ \mathcal{C}(\bu)^2 h^3.
\end{multline} 

\noindent {\bf 2.2. Smagorinsky nonlinear viscosity.}
We proceed to bound terms due to the
Smaginsky nonlinear viscosity, i.e. the fifth and sixth terms on the right
hand side of \eqref{eq:1st_bound}:
\[
(\hat \nu(\bu) \nabla \bu - \hat \nu(\bhu_h)
\nabla \bhu_h, \nabla \bldeta_\Pi)_Q - (\hat \nu(\bu) \nabla \bu, \nabla
\Pi \bldeta)_Q.
\]
Considering first the fourth term of the right hand side of
\eqref{eq:1st_bound} we have using \eqref{eq:holder} followed by
\eqref{eq:young} with $p=3$, $q=\tfrac32$. Then we use that $\|a
b\|_{L^{\frac32}(\Omega)} \leq \|a\|_{L^{\frac{15}{8}}(\Omega)} \|b\|_{L^{\frac{15}{2}}(\Omega)} $
\begin{align*}
&(\hat \nu(\bu) \nabla \bu - \hat \nu(\bhu_h)
\nabla \bhu_h, \nabla \bldeta_\Pi)_Q \leq  \delta^2 (2|\nabla
\bu| |\nabla \bldeta|+ |\nabla \bldeta|^2, |\nabla \bldeta_\Pi|)_Q 
\\
&\quad  \leq  \delta^2 \Bigl(\int_I \epsilon_0^{-\frac32} \|\nabla
\bu\|_{L^{\frac{15}{2}}(\Omega)}^{\frac32} \|\nabla
\bldeta_\Pi\|_{L^{\frac{15}{8}}(\Omega)}^{\frac32} ~\mbox{d}t 
+ (\epsilon_0^{3}+\epsilon_1^{\frac32})\|\nabla \bldeta\|_{L^3(Q)}^3 + \epsilon_1^{-3} \|\nabla
\bldeta_\Pi\|_{L^{3}(Q)}^3 \Bigr)
\end{align*}
for all $\epsilon_0,\, \epsilon_1 \in \mathbb{R}^+$. First observe
that using Sobolev injection,
\[
\|\nabla
\bu\|_{L^{\frac{15}{2}}(\Omega)} \|\nabla
\bldeta_\Pi\|_{L^{\frac{15}{8}}(\Omega)} \leq 
C \|\bu\|_{W^{2,3}(\Omega)} \|\nabla
\bldeta_\Pi\|_{\Omega} \leq C h \|\bu\|_{W^{2,3}(\Omega)}^2
\]
and by applying Lemma \ref{lem:lptol2},
\[
\|\nabla \bldeta_\Pi\|_{L^{3}(\Omega)} \leq \|\bu - \pi_h \bu\|_{W^{1,3}(\Omega)}
+ C h^{-\frac{1}{2}} \|\bldeta_{\Pi}\|_{1,\Omega} \leq C h (
|\bu|_{W^{2,3}(\Omega)} +  h^{-\frac{1}{2}} |\bu|_{H^{2}(\Omega)} ).
\]
Collecting these bounds and choosing $\epsilon_0$ and $\epsilon_1$
small so that
\[
(\epsilon_0^{3}+\epsilon_1^{\frac32}) \leq \frac{1}{16} 
\]
we have the bound
\begin{equation}
(\hat \nu(\bu) \nabla \bu - \hat \nu(\bhu_h)
\nabla \bhu_h, \nabla \bldeta_\Pi)_Q \leq \frac{1}{16} \delta^2
\|\nabla \bldeta\|_{L^3(Q)}^3 
+ C h^{\frac72} \|\bu\|^3_{L^3(I;W^{2,3}(\Omega))}.
\end{equation}
Now we consider the fifth term in the right hand side of \eqref{eq:1st_bound},
that quantifies the consistency error. We note that by partial integration and
Cauchy-Schwarz inequality followed by Young's inequality,
\eqref{eq:young}  and the stability of the $L^2$-projection $\bPi$, we have
\begin{align*}
- (\hat \nu(\bu) \nabla \bu, \nabla
\bPi \bldeta)_\Omega &= (\hat \nu(\bu) \nabla \bu, \nabla(\bldeta-
\bPi \bldeta))_\Omega+  (\nabla \cdot \nu(\bu) \nabla \bu,
\bldeta)_\Omega +   (\nu(\bu) \nabla \bu \cdot \bn,
t \bhu_h)_{\partial \Omega}  
\\
&\leq \frac14 \tau_L^{-1} \|\bldeta\|_\Omega^2 + C \tau_L \delta^4
\|\nabla \cdot |\nabla \bu|_F \nabla \bu\|_\Omega^2 
+Ch^3 \|\nabla \bu\|_{W^{2,3}(\Omega)}^3 
\\
&\qquad 
+ Ch^3 (\gamma_1 \max( \mu h^{-1} ,
U))^{-1}  \|\nabla \bu\|_{W^{2,3}(\Omega)}^4 + \frac{1}{16} s_1(\bhu_h,\bhu_h),
\end{align*}
where we used, that by the Cauchy-Schwarz inequality and \eqref{eq:L2proj2}
there holds
\[
 (\hat \nu(\bu) \nabla \bu, \nabla(\bldeta-
\bPi \bldeta))_\Omega \leq \delta^2 h \|\nabla
\bu\|^2_{L^4(\Omega)}|\bu|_{H^2(\Omega)} \leq \delta^2 h \|\nabla \bu\|_{W^{2,3}(\Omega)}^3 
\]
and by H\"olders inequality with $p=3/2$ and $q=3$ inequality followed
by a global trace inequality and the inverse inequality
\eqref{eq:Lpinv} on $\partial \Omega$,
\begin{align*}
(\nu(\bu) \nabla \bu \cdot \bn,
t \bhu_h)_{\partial \Omega} 
& \leq \delta^{\frac32} (\gamma_1 \max( \mu h^{-1} ,
U))^{-1}  \|\nabla \bu\|^2_{L^3(\partial \Omega)}\gamma_1
\max(\mu h^{-1} , U) \delta^{\frac12} \|t \bhu_h\|_{L^3(\partial \Omega)}
\\
&\leq C h^3 (\gamma_1 \max( \mu h^{-1} ,
U))^{-2}  \|\bu\|_{W^{2,3}(\Omega)}^4 + \frac{1}{16} s_1(\bhu_h,\bhu_h).
\end{align*}
This completes the bound of the Smagorinsky terms. Collecting the
different contributions we get
\begin{multline}\label{eq:smag_bound}
(\hat \nu(\bu) \nabla \bu - \hat \nu(\bhu_h)
\nabla \bhu_h, \nabla \bldeta_\Pi)_Q - (\hat \nu(\bu) \nabla \bu, \nabla
\bPi \bldeta)_\Omega\\
\leq \frac14 \tau_L^{-1} \|\bldeta\|_\Omega^2 + \frac{1}{16} \delta^2
\|\nabla \bldeta\|_{L^3(Q)}^3 
+ \frac{1}{16}
\int_I s_1(\bhu_h,\bhu_h) ~\mbox{d}t + \mathcal{C}(\bu)^2 h^3.
\end{multline}
\noindent {\bf 2.3. Terms related to boundary conditions and stabilization.}
These are the last three terms of the right hand side of \eqref{eq:1st_bound}.
\[
\frac12\|\bldeta_{\Pi}(0)\|_\Omega^2 +  \int_I (bc(\bldeta,\bPi \bldeta)+ s(\bhu_h,\bPi \bu)) ~\mbox{d}t.
\]
The first term, related to approximation of initial data is bounded
using approximation
\[
\|\bldeta_{\Pi}(0)\|_\Omega^2 \lesssim h^4 \|\bu(0)\|_{H^2(\Omega)}^2.
\]
For the $bc$ form related to boundary conditions we apply the second
inequality of Lemma \ref{lem:bc_stab} to obtain
\begin{align*}
\int_I (bc(\bldeta,\bPi \bldeta) &\leq C \gamma_1^{-1} \|\mu^{\frac12}
\nabla \bldeta\|_Q^2 + C \mu h^2
|\bu|_{L^{\infty}(I;H^2(\Omega))}^2 
\\
&\qquad + \frac14 \int_I s_1(\bhu_h,\bhu_h) ~\mbox{d}t +
\frac14 s_1(\bPi \bu,\bPi \bu).
\end{align*}
For the stabilization term finally we proceed using the Cauchy-Schwarz
inequality followed by the arithmetic geometric inequality
\[
\int_I s(\bhu_h,\bPi \bu)) ~\mbox{d}t \leq \frac18 \int_I
s(\bhu_h,\bhu_h) ~\mbox{d}t + 2 \int_I s(\bPi \bu,\bPi \bu) ~\mbox{d}t.
\]
Applying Lemma \ref{lem:stab_consist} we conclude that
\begin{multline}\label{eq:bc_stab_bound}
\frac12\|\bldeta_{\Pi}(0)\|_\Omega^2 +  \int_I (bc(\bldeta,\bPi
\bldeta)+ s(\bhu_h,\bPi \bu)) ~\mbox{d}t \\
\leq \frac{1}{16} \delta^2 \|\nabla \bldeta\|_{L^3(Q)}^3 + C \gamma_1^{-1} \|\mu^{\frac12}
\nabla \bldeta\|_Q^2 + \frac18 \int_I
s(\bhu_h,\bhu_h) ~\mbox{d}t  + \frac14 \int_I s_1(\bhu_h,\bhu_h)
~\mbox{d}t + \mathcal{C}(\bu)^2 h^3.
\end{multline}
\paragraph{{\bf Step 3. Application of scale separation argument}.}
Collecting the above bounds \eqref{eq:convect}, \eqref{eq:smag_bound} and \eqref{eq:bc_stab_bound} and applying them to \eqref{eq:1st_bound} we get, for
$\gamma_1$ sufficiently large 
\begin{multline}\label{eq:before_Gronwall}
\frac12\|\bldeta(T)\|_\Omega^2 + \frac18 \|\mu^{\frac12} \nabla \bldeta \|_Q^2 +\frac{1}{16} \delta^2
\|\nabla \bldeta(t)\|_{L^3(Q)}^3 + \frac{1}{16} \int_I s(\bhu_h,\bhu_h) ~\mbox{d}t\\
\leq \tau_L^{-1}  \|\bldeta\|_Q^2 ~\mbox{d}t - ((\bldeta \cdot \nabla)
\bu,\bldeta)_Q  + \mathcal{C}(\bu)^2 h^3.
\end{multline}
A naive application of Gronwall's lemma at this stage results in exponential growth
with a coefficient proportional to the maximum of the velocity gradient.
Instead recall the inequality proven in Lemma \ref{lem:conv_term}, 
\[
 ((\bldeta\cdot \nabla ) \bu, \bldeta)_{Q} \leq   
 \frac{\epsilon}{ 2}\|\mu^{\frac12} \nabla \bldeta\|_{Q}^2
 + \frac{\epsilon}{ 2} \nu(\delta) \|\nabla
 \bldeta\|_{L^3(Q)}^{3} + \|(\epsilon^{-1} \tau^{-1}_L+|\nabla \bar \bu|_F)^{\frac12} \bldeta\|^2_{Q}.
\]
Applying this inequality, with $\epsilon=1/16$ in the right hand side
of \eqref{eq:before_Gronwall} we obtain
\begin{equation}\label{eq:last_bound}
\|\bldeta(T)\|_\Omega^2 + \frac18 \mu \|\nabla \bldeta\|_{Q}^2+ \frac{1}{16} \delta^2
\|\nabla \bldeta\|_{L^3(Q)}^3 + \frac18\int_I s(\bhu_h,\bhu_h) ~\mbox{d}t 
\leq
%
 2 \int_I (17 \tau_L^{-1} + |\nabla \bar \bu|_F)
 \|\bldeta\|_\Omega^2 ~\mbox{d}t +\mathcal{C}(\bu)^2 h^{3}
\end{equation}
where 
\begin{multline*}
\mathcal{C}(\bu)^2 \lesssim U | \bu|_{L^2(I;H^{2}(\Omega))}^2 +U^3 +
h (\tau_L \|\nabla \times \bu\|_{L^2(I;L^\infty(\Omega))}
+1)|\bu|^2_{L^\infty(I;H^2(\Omega))} \\
+
\|\bu\|^3_{L^3(I;W^{2,3}(\Omega))}+\|\bu\|^4_{L^4(I;W^{2,3}(\Omega))}+ h
\|\nabla \cdot |\nabla \bu|_F \nabla \bu\|^2_Q.
\end{multline*}
Assuming $U$ bounded by $\|\bu\|_{L^3(I;W^{2,3}(\Omega))}$, we can absorb the first two terms in the
right hand side in the $\|\bu\|^3_{L^3(I;W^{2,3}(\Omega))}$ term.
\paragraph{{\bf Step 4.  Application of Gronwall's Lemma}.}
Applying Gronwall's inequality on integral form \cite{bellman1943} to \eqref{eq:last_bound} we see that
for all $t \in I$
\[
\|\bldeta(t)\|_\Omega^2 \leq e^{36 (t/\tau_L)} \mathcal{C}(\bu)^2 h^{3}
\]
and as a consequence, using \eqref{eq:last_bound},
\begin{align*}
\frac18 \mu \|\nabla \bldeta\|_{Q}^2+ \frac{1}{16} \delta^2
\|\nabla \bldeta\|_{L^3(Q)}^3 + \frac18\int_I s(\bhu_h,\bhu_h) ~\mbox{d}t  
&\leq 2 \int_I (17 \tau^{-1}_L + \|\nabla \bar
\bu\|_{L^\infty(\Omega)}) e^{36 (t/\tau_L)} ~\mbox{d}t\,\mathcal{C}(\bu)^2 h^{3}
\\
&\leq C T \tau^{-1}_L  e^{36(T/\tau_L)} \mathcal{C}(\bu)^2h^{3}
\end{align*}
leading to the bound,
\begin{align}
&\sup_{t \in I} \|(\bu - \bhu_h)(t)\|^2_\Omega + \frac18 T^{-1} \tau_L \mu \|\nabla (\bu - \bhu_h)\|_{Q}^2
\\ \nonumber
&\qquad + T^{-1} \tau_L \delta^2 \|\nabla (\bu -
\bhu_h)\|^{3}_{L^3(Q)} + T^{-1} \tau_L \int_I
s(\bhu_h,\bhu_h)~\mbox{d}t 
\lesssim 
e^{36(T/\tau_L)}\mathcal{C}(\bu)^2 h^{3}.
\end{align}
This concludes the proof.
\end{proof}
\begin{corollary}\label{cor:inftybound}
Under the same assumption as for Theorem \ref{thm:error_bound}, the solution to \eqref{eq:FEM} satisfies the following bound
\[
\|\bhu_h \|_{L^\infty(Q)} \lesssim \mathcal{C}(\bu)(1+ e^{18 (T/\tau_L)}).
\]
\begin{proof}
Observe that under the regularity assumptions of Theorem \ref{thm:error_bound} $\bu \in
L^\infty(Q)$. It then follows that
\[
\|\bhu_h\|_{L^\infty(Q)} \leq \|\bhu_h -\pi_h \bu\|_{L^\infty(Q)} + C
\|\bu\|_{L^\infty(Q)} \leq C h^{-\frac{d}{2}} \|\bhu_h -\pi_h \bu\|_{L^\infty(I;L^2(\Omega))}+ C
\|\bu\|_{L^\infty(Q)}.
\]
The claim follows using that $\|\bu\|_{L^\infty(Q)} \lesssim
\|\bu\|_{L^\infty(I;H^2(\Omega))}$ and
\begin{align*}
\|\bhu_h -\pi_h \bu\|_{L^\infty(I;L^2(\Omega))} &\leq \|\bhu_h -
\bu\|_{L^\infty(I;L^2(\Omega))} + \|\bu -\pi_h
\bu\|_{L^\infty(I;L^2(\Omega))} 
\\
&\lesssim \mathcal{C}(\bu)  e^{9 (T/\tau_L)} h^{\frac32} + h^2 |\bu|_{L^\infty (I; H^2 (\Omega))}.
\end{align*}
\end{proof}
\end{corollary}

\section{Numerical examples}

In the numerical examples below we use the minimal compatible
element of \cite{CH18} with piecewise affine velocity and piecewise
constant pressure  on macro elements (see \cite{BCH20} for implementation details and application to linear incompressible problems), together with the backwards differentiation formula BDF2 for time stepping.
On each time step we use a linearized formula, using the solution from the previous time step, and solve only once.
We use only Smagorinsky stabilization, i.e., we set $\gamma_0=\gamma_1=0$ in (\ref{eq:FEM}). The Smagorinsky term is
set by $\nu := \gamma\, \vert T\vert  \, |\nabla \bu|_F$ where $\vert T\vert$ is the element area and $\gamma$ controls the amount of dissipation in the model.

\subsection{Shear layers}

We shall first consider an example using the Euler equations, i.e., $\mu=0$ in (\ref{eq:NS_smag}), the double shear layer problem \cite{BCG89}. 
The computational domain is $\Omega) =(0,2\pi)\times(0,2\pi)$ and initial conditions
\begin{equation}
u_y(\bx,0) = \delta \sin{(x)}, \quad u_x(\bx,0) = \left\{\begin{array}{>\displaystyle{l}} \tanh{((y-\pi/2)/\rho)},\quad y\leq \pi\\[3mm]  \tanh{((3\pi/2-y)/\rho)},\quad y> \pi \end{array}\right. 
\end{equation}
which gives two horizontal shear layers perturbed by a small vertical velocity. We take $\rho=\pi/15$ and $\delta=0.05$ and apply periodic boundary conditions.
The problem is solved on a Union Jack (macro element) mesh of $100\times 100$ (boundary) nodes using a timestep size $k=1/100$.

In Figs. \ref{fig:vortnostab}--\ref{fig:vortminus2} we show the effect of dissipation on vorticity  by increasing $\gamma$ from zero to $10^{-1}$.
Th unstabilized solution shows oscillations in the solution which worsen with time, whereas the choice $\gamma=10^{-1}$ compares well with \cite{BCG89}.

\subsection{Vortex shedding} The second example shows the effect of dissipation on von Karman vortex shedding around a cylider.
We used $\mu = 3\times 10^{-4}$ (which is close to the limit for stable solutions with $\gamma=0$ on this mesh). The outer domain is 
$(-1/2,2)\times (-1/,1/2)$, and the cylinder has center at the origin and radius $r=1/10$. The boundary conditions are $\bu = \bf 0$ at $y=\pm 1/2$ and at the cylinder, homogeneous Neumann conditions at $x =2$, and $\bu =(3/2-6 y^2,0)$ at $x=-1/2$. The initial conditions correspond to the stationary Stokes solution,
and the timestep size $k=1/100$. The computational (macro) mesh is shown in Fig. \ref{fig:meshvort}, and the solution at time $t=10$ is shown in Fig. \ref{fig:vorticity} for increasing $\gamma$. We note that too much dissipation severely affects the shape of the vortex street, whereas smaller amounts of dissipation only serve to stabilize the solution. 

In Fig. \ref{fig:velunstab} we show that the method can be unstable for high Reynolds numbers. We used $\mu=10^{-6}$ and show the instability evolving for $t=0.15$, $t=0.2$ and $t=0.3$ with $\gamma=0$. In Fig. \ref{fig:velstab} we show the corresponding stabilized solution with $\gamma=10^{-1}$. The velocities are shown in the nodes of the macro mesh. Finally, in Fig. \ref{fig:vort1e-6} we show the relative streamlines at time $t=10$ for the stabilized model.

\section{Conclusion}
We have considered a p-Laplacian Smagorinsky model for high
Reynolds flow problems. We showed using a scale separation
argument that the Smagorinsky model has stability properties that only
depend on the large scales of the flow, provided a spectral gap exists
for the particular flow configuration. The set of
non-essential fine scales grows as the perturbation due to the
Smagorinsky model increases, hence moderating the exponential growth.
In a second part we considered the Smagorinsky model as a stabilizing
term in a low order finite element method and we showed that the resulting
stabilized method has optimal properties for smooth solutions in the
laminar high Reynolds number regime. 
The stabilized finite element method also inherits the reduced exponential growth through scale
separation from the continuous case. 

To the best of our knowledge
these results are the first that give quantitative evidence that the
Smagorinsky model enhances both the accuracy and the stability for computations of high Reynolds number flows.
 
\section*{Appendix}
\subsection*{Proof of Lemma \ref{lem:bc_stab}}
The first inequality is a classical bound for Nitsche's method. It
follows using the Cauchy-Schwarz inequality, followed by
\eqref{eq:young} with $p=q=2$, the trace
inequality \eqref{trace_H1} on each element face subset of $\partial
\Omega$ together with \eqref{inverse}
\begin{multline*}
bc(\bv_h, \bv_h) \leq 2 \gamma_1^{-1} h \|\mu^{\frac12} \nabla
\bv_h \|^2_{\partial \Omega}  + \frac18 \gamma_1 \max(\mu h^{-1},U)
\|t \bhu_h\|^2_{\partial \Omega} \\
\leq  C  \gamma_1^{-1} \|\mu^{\frac12} \nabla
\bv_h\|^2_{\Omega}+ \frac18 \gamma_1 \max(\mu h^{-1},U) \|t \bhu_h\|^2_{\partial
  \Omega}.
\end{multline*}
The second inequality follows from this result by adding and
subtracting the projection $\pi_h \bv$,
\begin{equation*}
bc(\bv, \bv_h)  = bc(\pi_h \bv, \bv_h) + bc(\bv - \pi_h \bv, \bv_h)= I
+ II.
\end{equation*}
Using the same previous inequality followed by the $H^1$-stability of $\pi_h$ we have
\[
I \leq C  \gamma_1^{-1} \|\mu^{\frac12} \nabla\bv\|^2_{\Omega}+ \frac18 \gamma_1 \max(\mu h^{-1},U) \|t \bv_h\|^2_{\partial
  \Omega}.
\]
Similar arguments as before also show that
\[
II \leq C  \gamma_1^{-1} (\|\mu^{\frac12} \nabla
(\bv - \pi_h \bv)\|^2_{\Omega} + h^2 \mu \|\bv\|_{H^2(\Omega)}^2+ \frac18 \gamma_1 \max(\mu h^{-1},U) \|t \bv_h\|^2_{\partial
  \Omega}.
\]
We conclude by using approximation in the first term of the right hand
of this inequality
side and then sum the bounds for $I$ and $II$.

\section*{Acknowledgement}
EB was partially supported by the EPSRC grants  EP/P01576X/1 and EP/T033126/1. PH was partially supported by the Swedish Research
Council Grant No.\  2018-05262. ML was partially suported by the Swedish Research
Council Grant No.\  2017-03911 and the Swedish Research Programme
Essence.
On behalf of all authors, the corresponding author states that there is no conflict of interest. 
\bibliographystyle{abbrv}
\bibliography{references}
\newpage
\begin{figure}[ht]
	\begin{center}
		\includegraphics[scale=0.15]{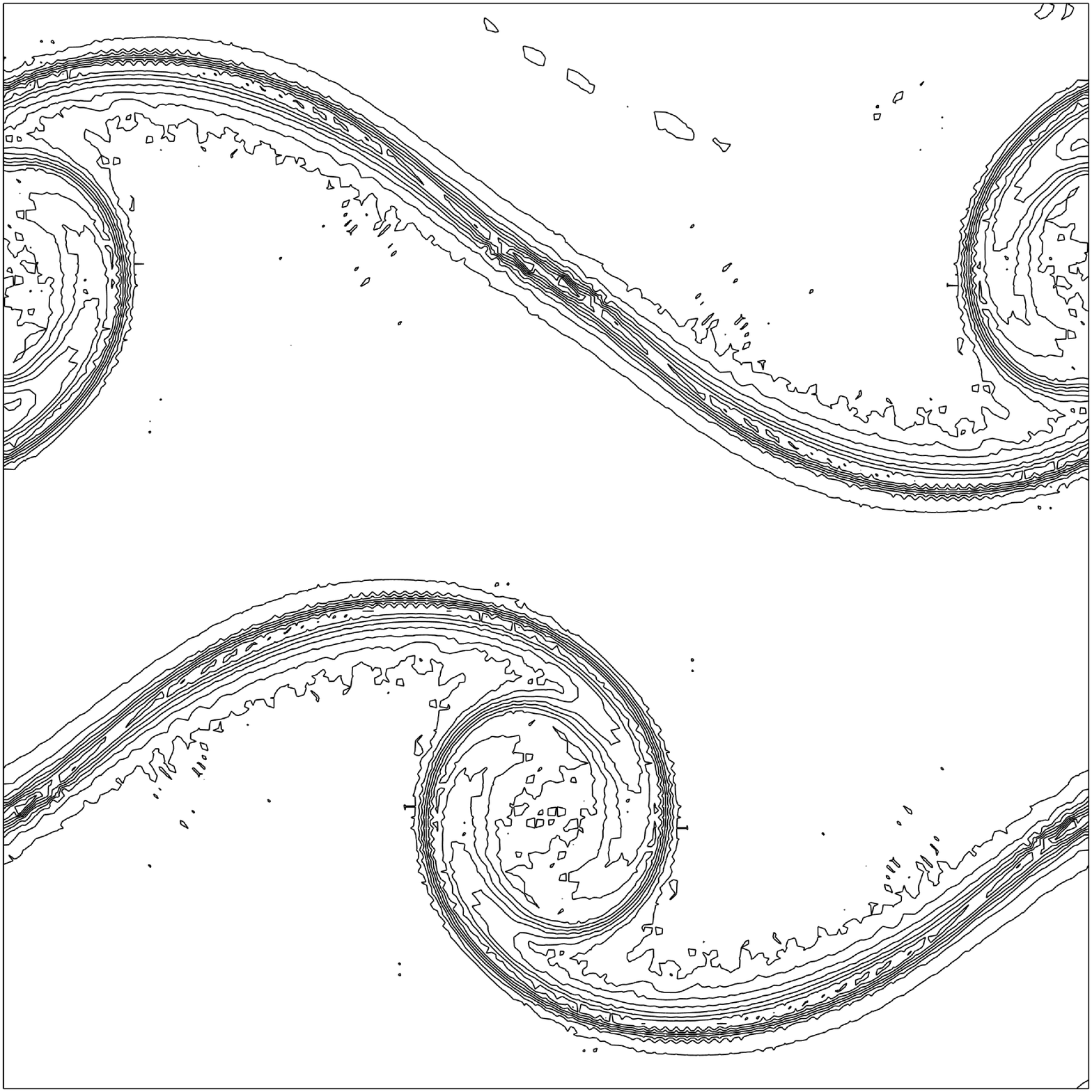}\includegraphics[scale=0.15]{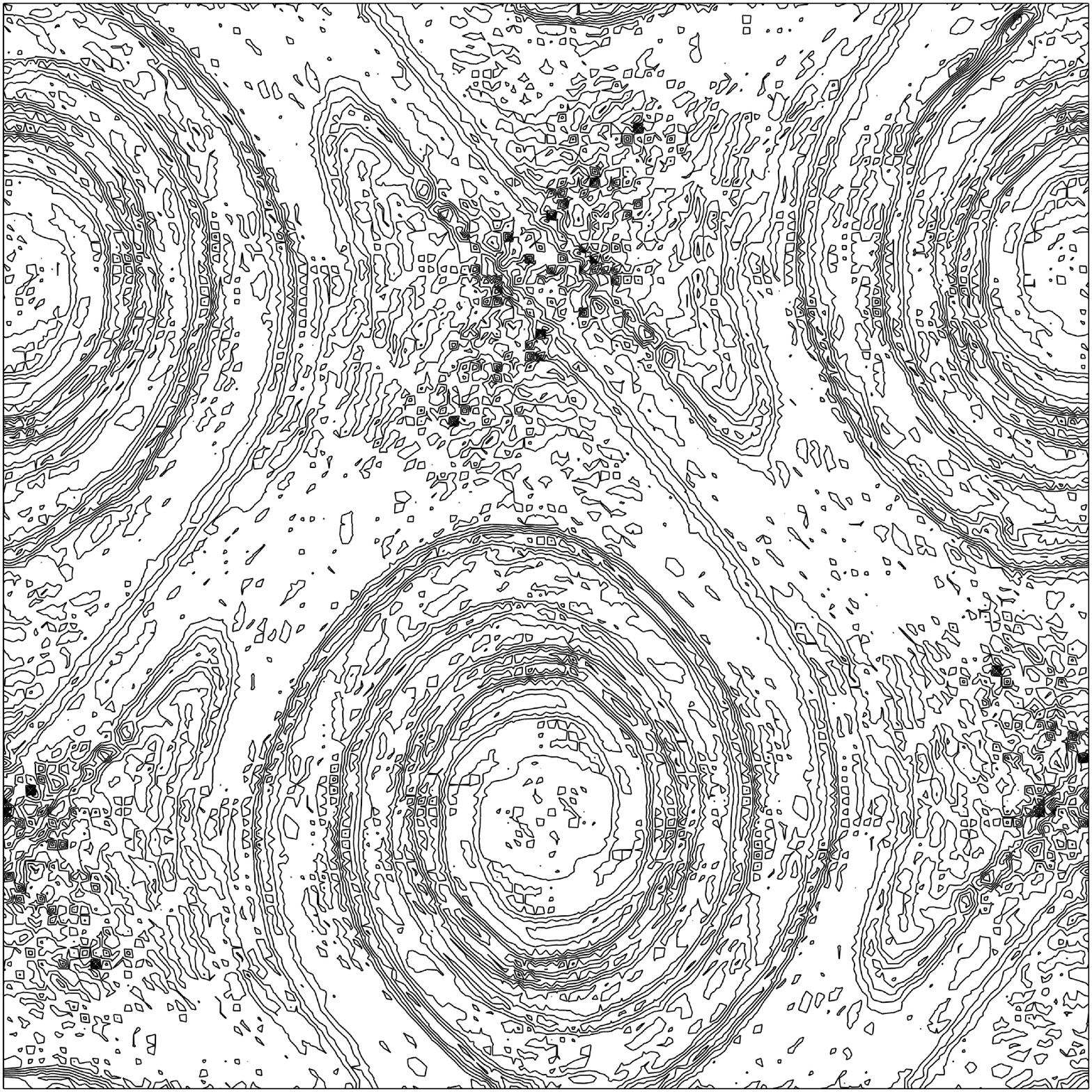}
	\end{center}
	\caption{Vorticity contours for the double shear layer at time $t=6$ and $t=12$, $\gamma=0$.}
	\label{fig:vortnostab}
\end{figure}
\begin{figure}[ht]
	\begin{center}
		\includegraphics[scale=0.15]{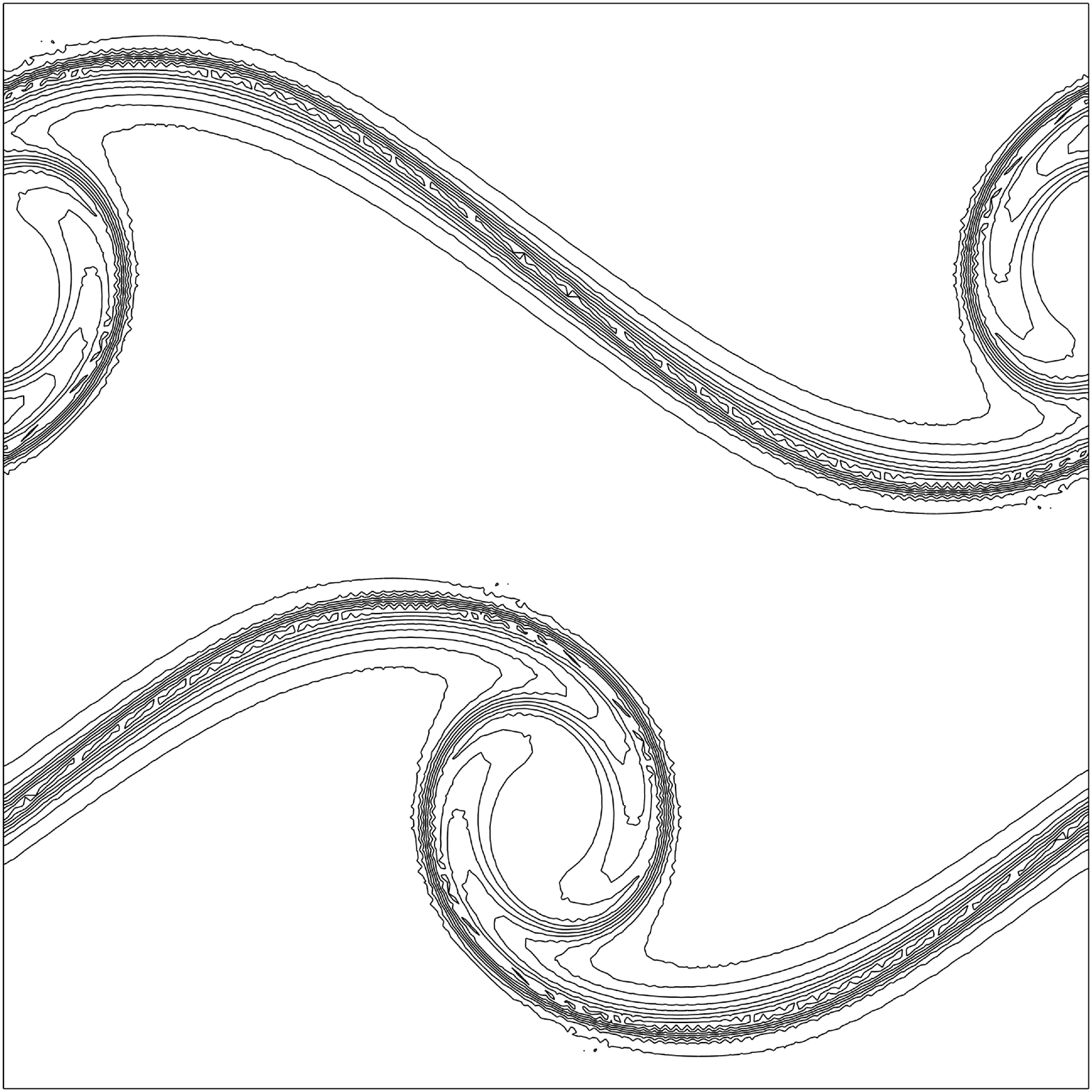}\includegraphics[scale=0.15]{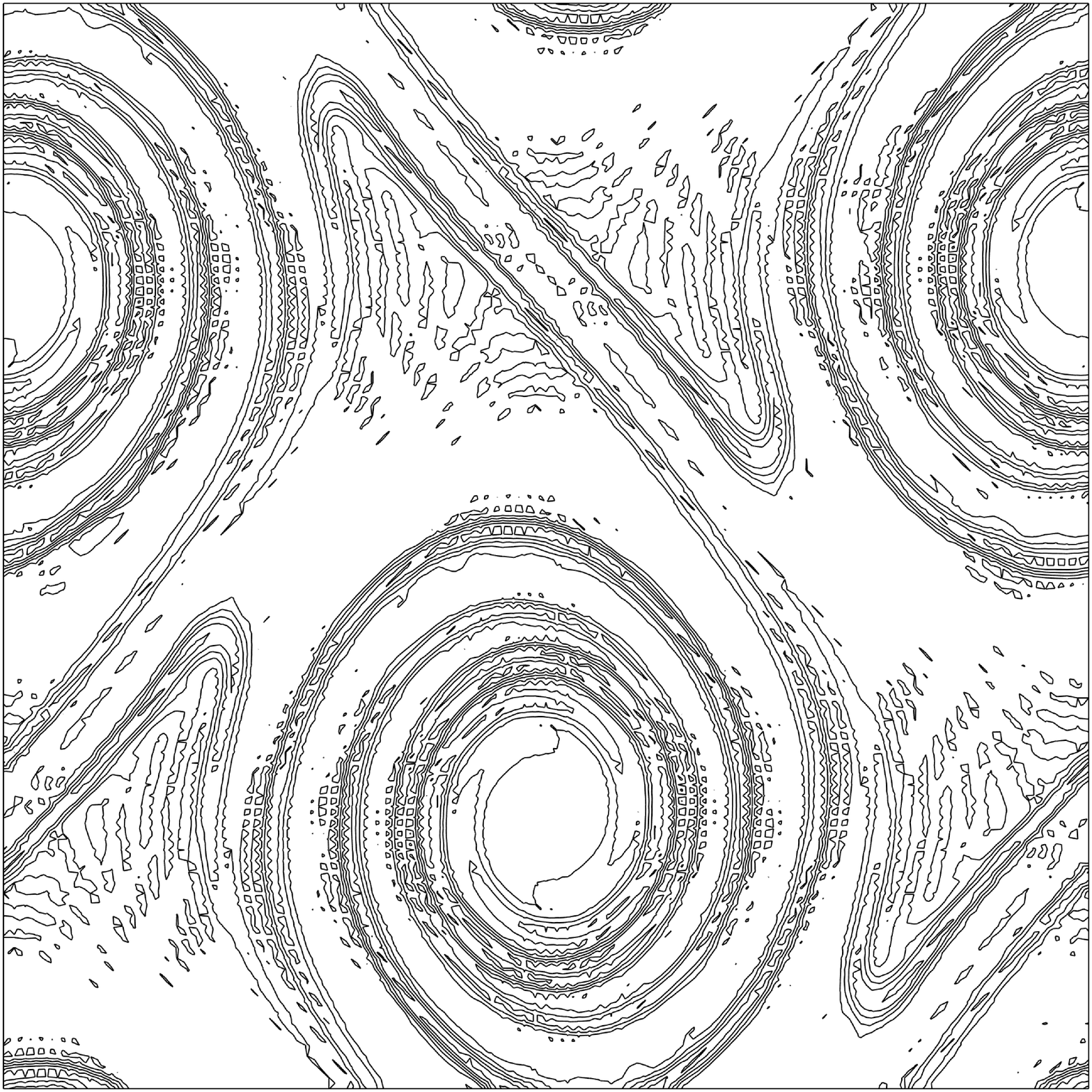}
	\end{center}
	\caption{Vorticity contours for the double shear layer at time $t=6$ and $t=12$, $\gamma=10^{-2}$.}
	\label{fig:vortminus1}
\end{figure}

\begin{figure}[ht]
	\begin{center}
		\includegraphics[scale=0.15]{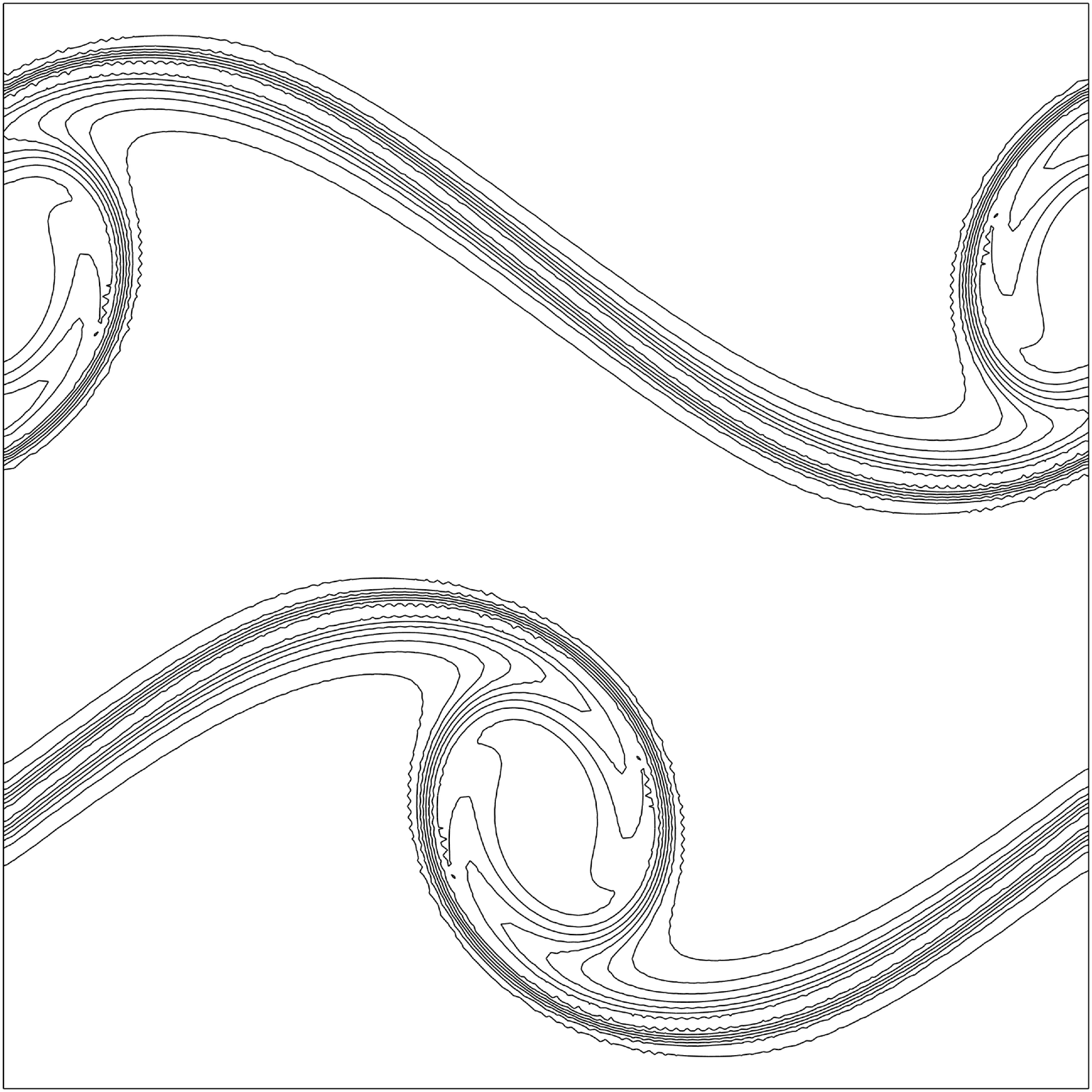}\includegraphics[scale=0.15]{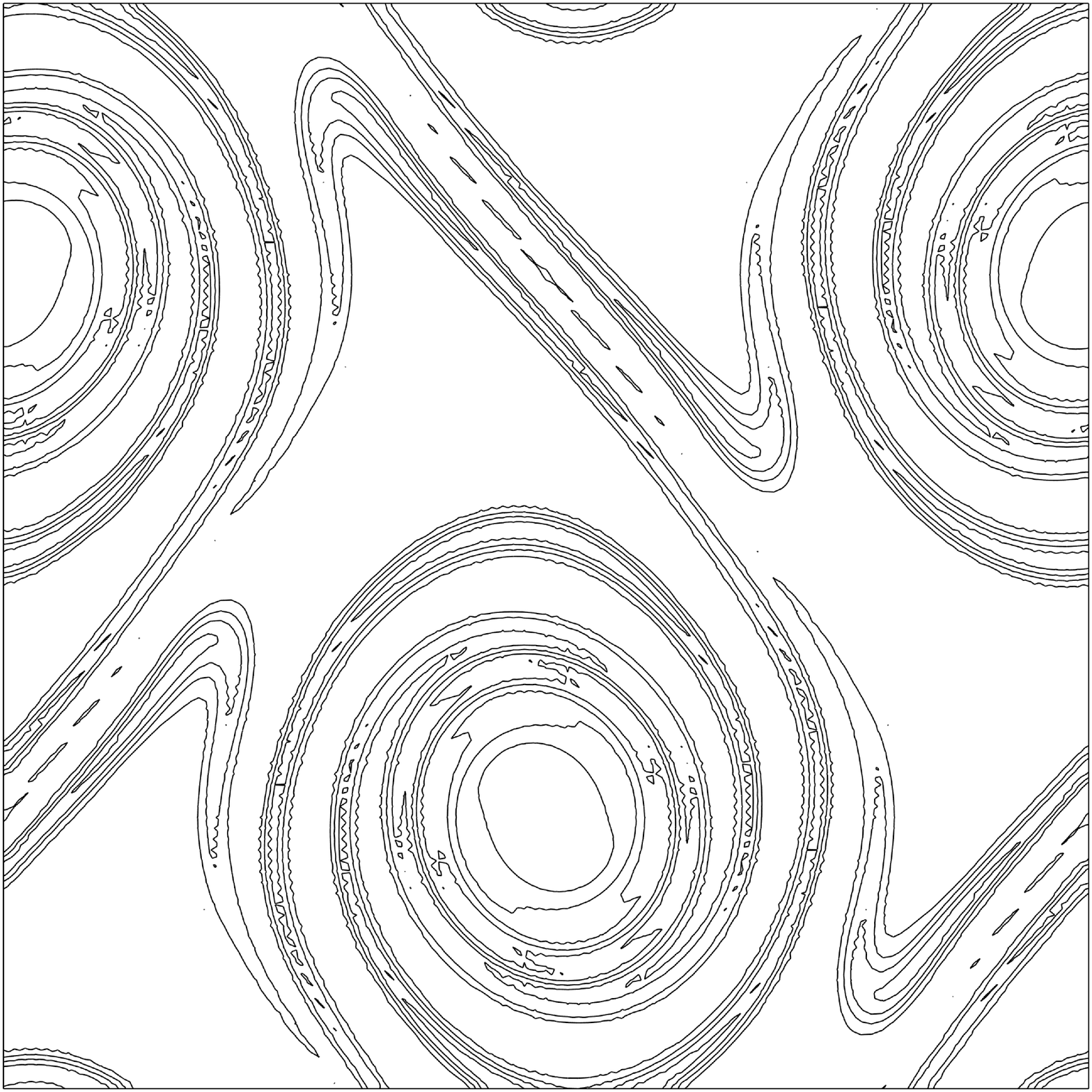}
	\end{center}
	\caption{Vorticity contours for the double shear layer at time $t=6$ and $t=12$, $\gamma=10^{-1}$.}
	\label{fig:vortminus2}
\end{figure}

\begin{figure}[ht]
	\begin{center}
		\includegraphics[scale=0.35]{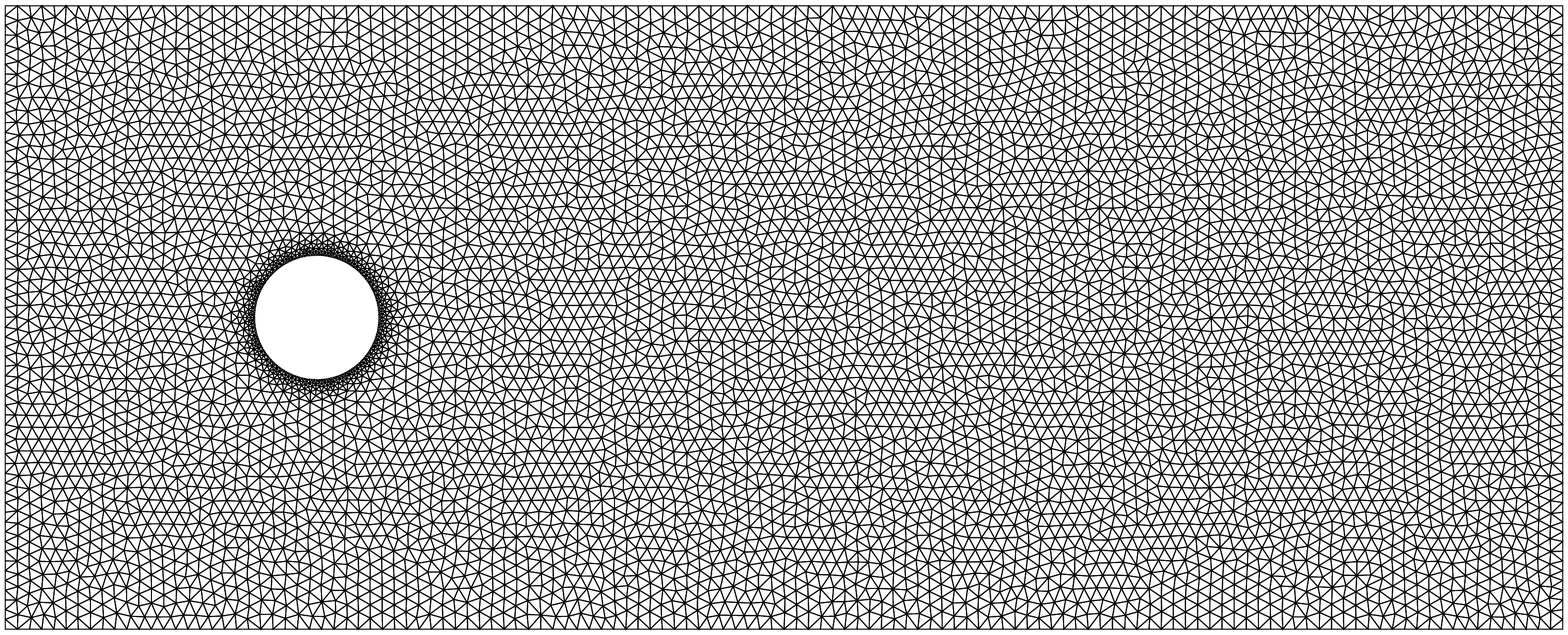}
	\end{center}
	\caption{Mesh used for vortex shedding computation.}
	\label{fig:meshvort}
\end{figure}

\begin{figure}[ht]
\centering
  \begin{tabular}{@{}c@{}}
   \includegraphics[scale=0.2]{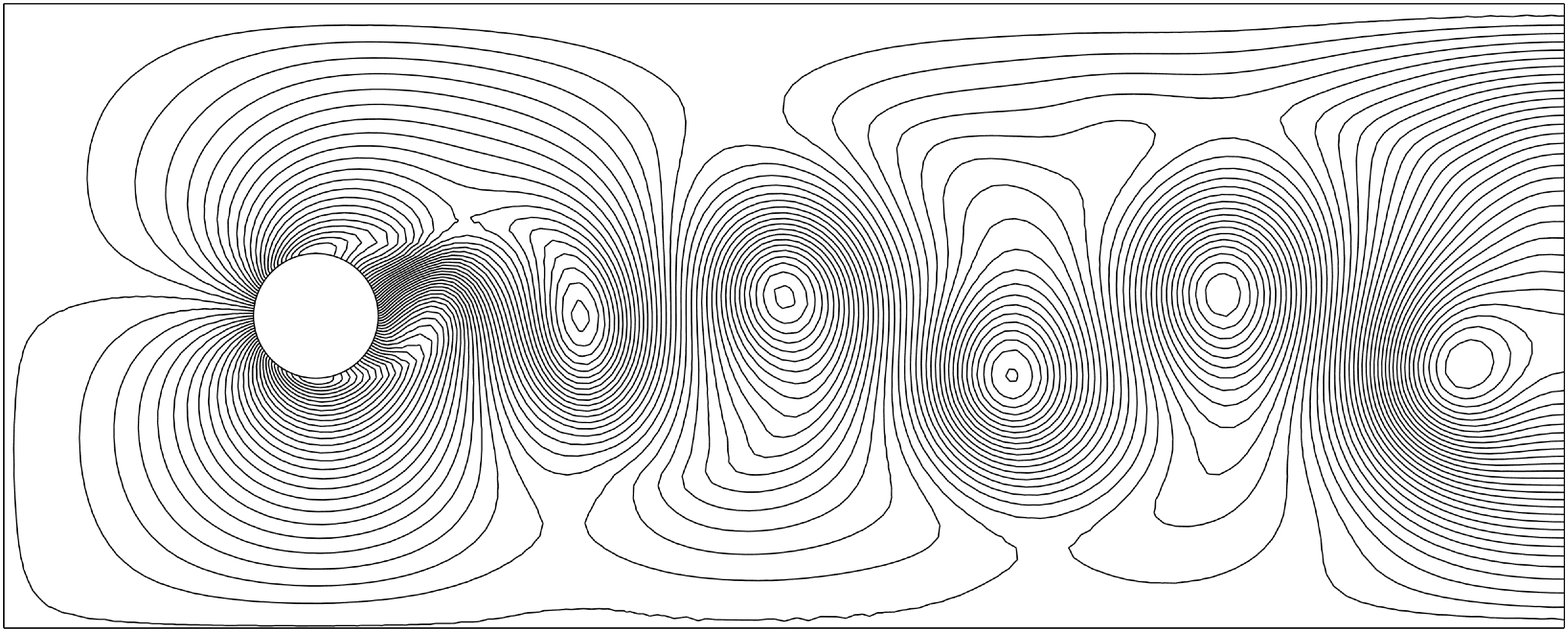}\\
  \end{tabular}
\vspace{-3cm}

  \begin{tabular}{@{}c@{}}
   \includegraphics[scale=0.2]{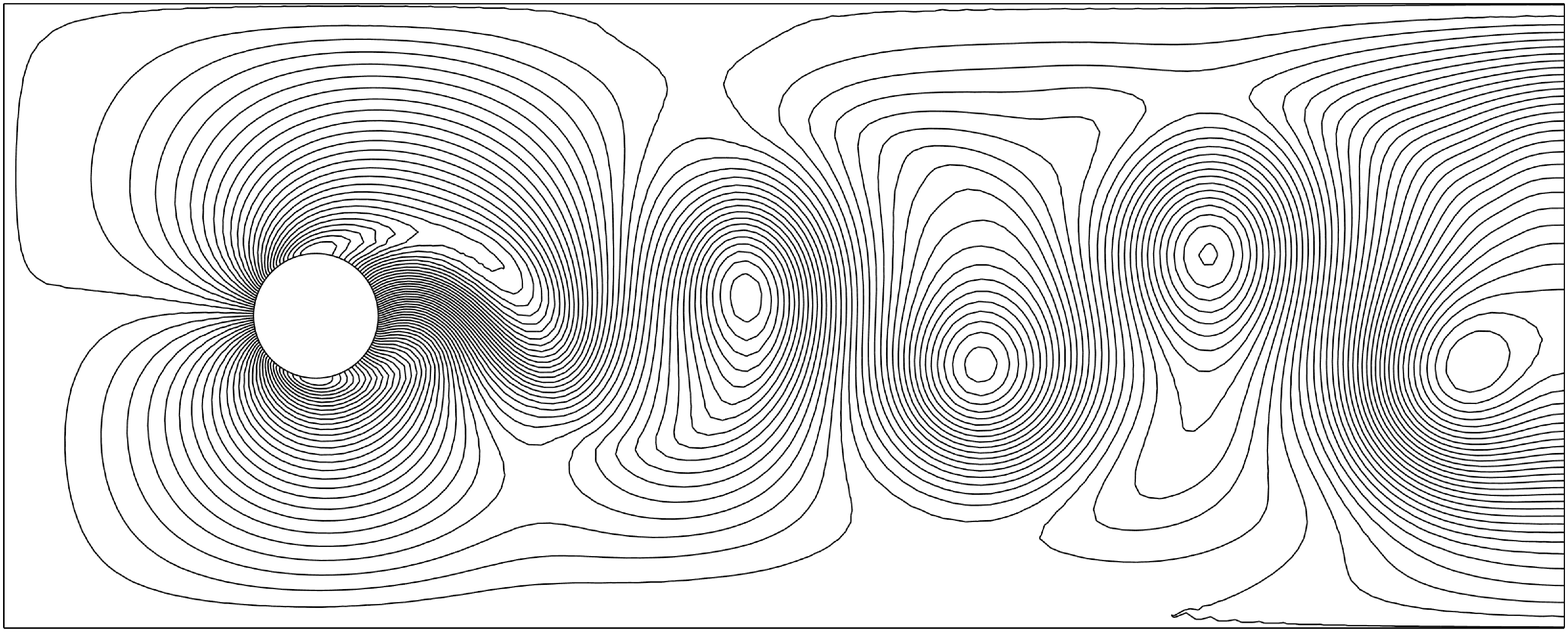}\\
  \end{tabular}

    \vspace{-3cm}

  \begin{tabular}{@{}c@{}}
   \includegraphics[scale=0.2]{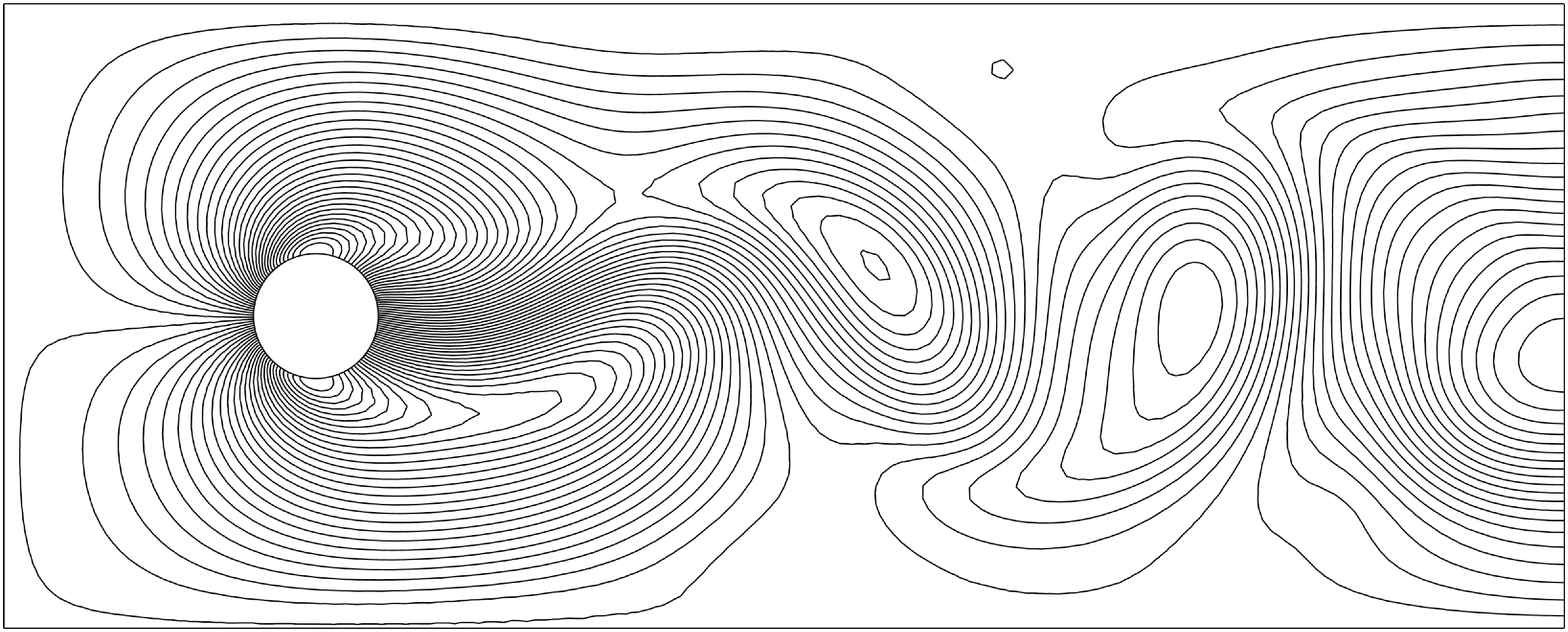}\\
  \end{tabular}
		
	\caption{Relative streamlines at time $t=10$, $\gamma=0$,$\gamma=10^{-1}$, $\gamma=1$ from top.}
	\label{fig:vorticity}
\end{figure}

 \begin{figure}[ht]
 \begin{tabular}{@{}c@{}}
   \includegraphics[scale=0.2]{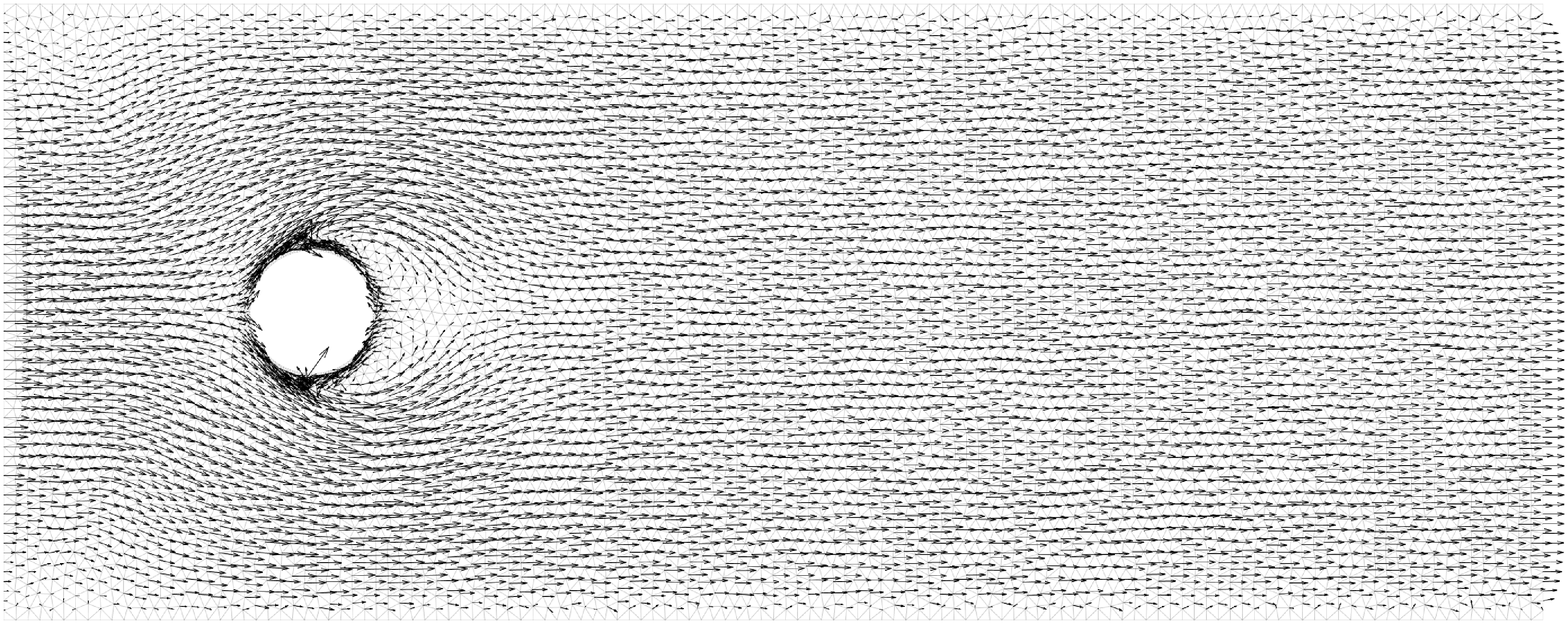}\\
  \end{tabular}
  
\vspace{-1.5cm}

 \begin{tabular}{@{}c@{}}
   \includegraphics[scale=0.2]{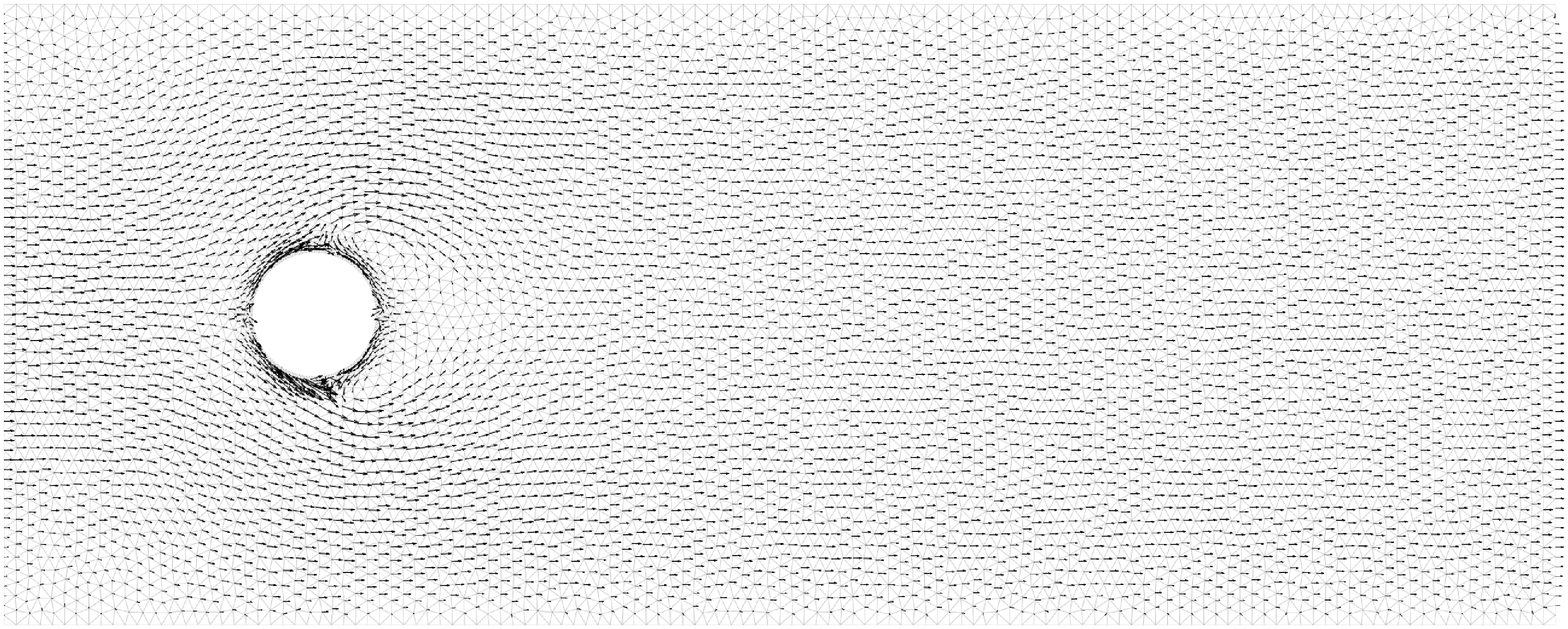}\\
  \end{tabular}
  
\vspace{-2cm}

 \begin{tabular}{@{}c@{}}
   \includegraphics[scale=0.2]{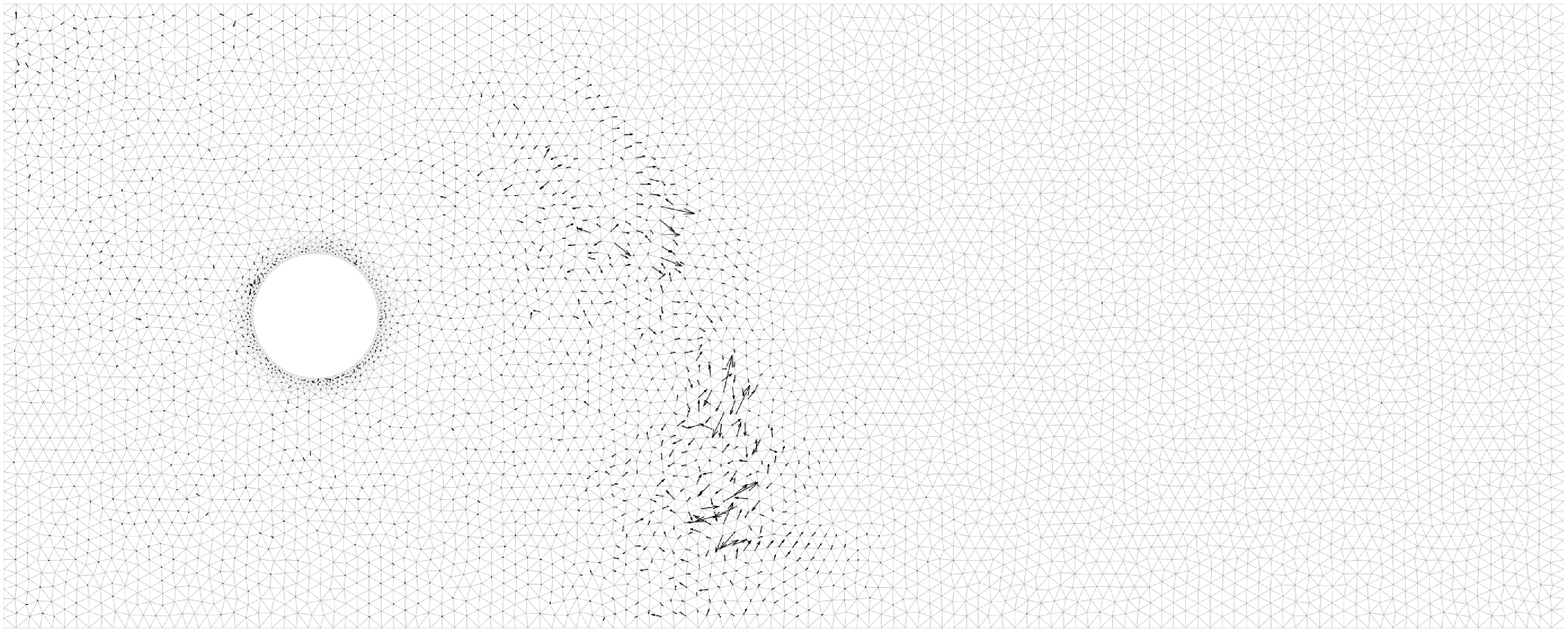}\\
  \end{tabular}
		
	\caption{Velocities after 15, 20, and 30 timesteps, from top, $\mu=10^{-6}$, $\gamma = 0$.}
	\label{fig:velunstab}
\end{figure}
 \begin{figure}[ht]
 \begin{tabular}{@{}c@{}}
   \includegraphics[scale=0.2]{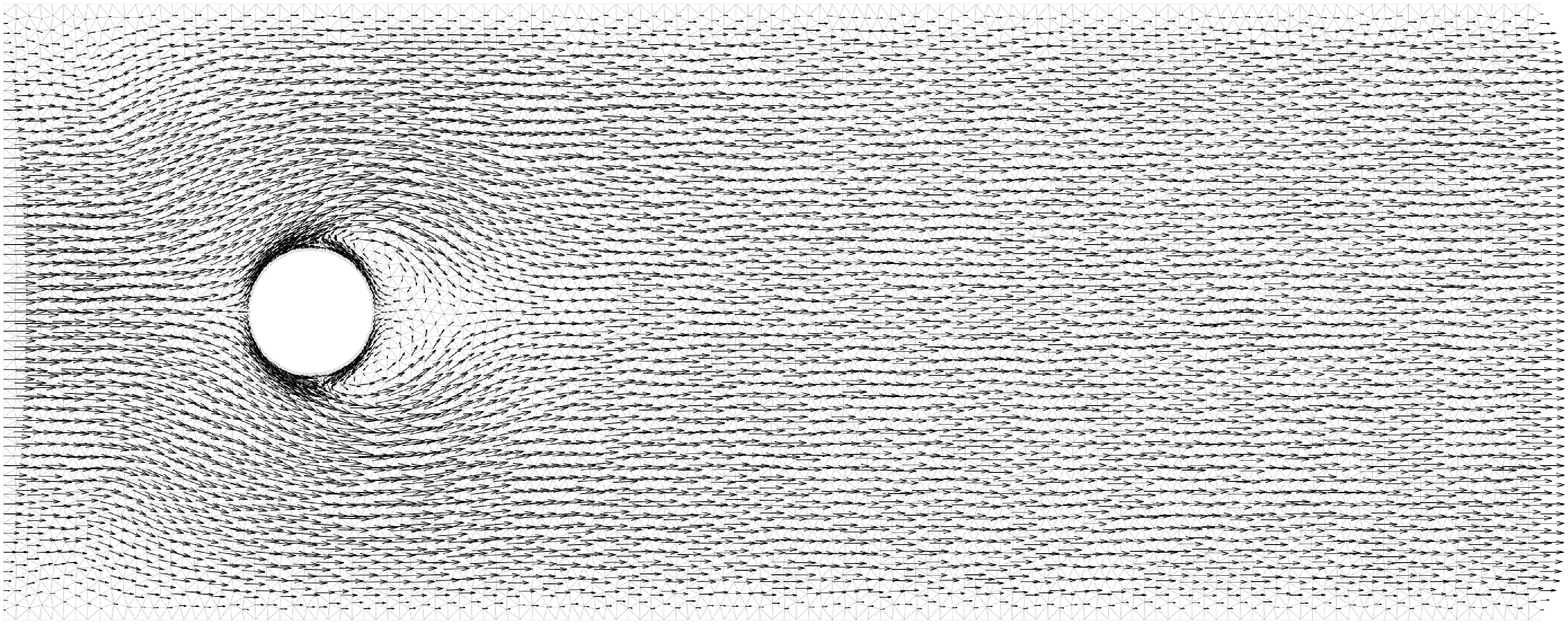}\\
  \end{tabular}
  
\vspace{-1.5cm}

 \begin{tabular}{@{}c@{}}
   \includegraphics[scale=0.2]{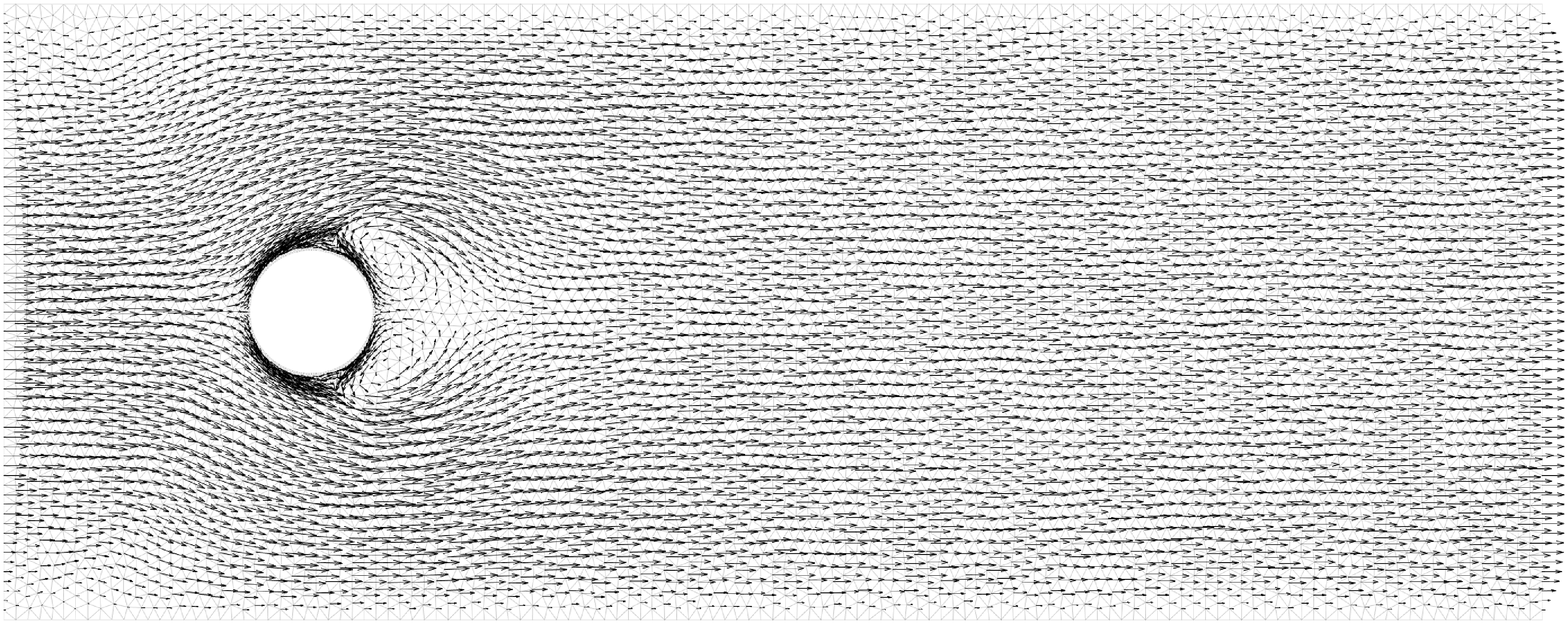}\\
  \end{tabular}
  
\vspace{-2cm}

 \begin{tabular}{@{}c@{}}
   \includegraphics[scale=0.2]{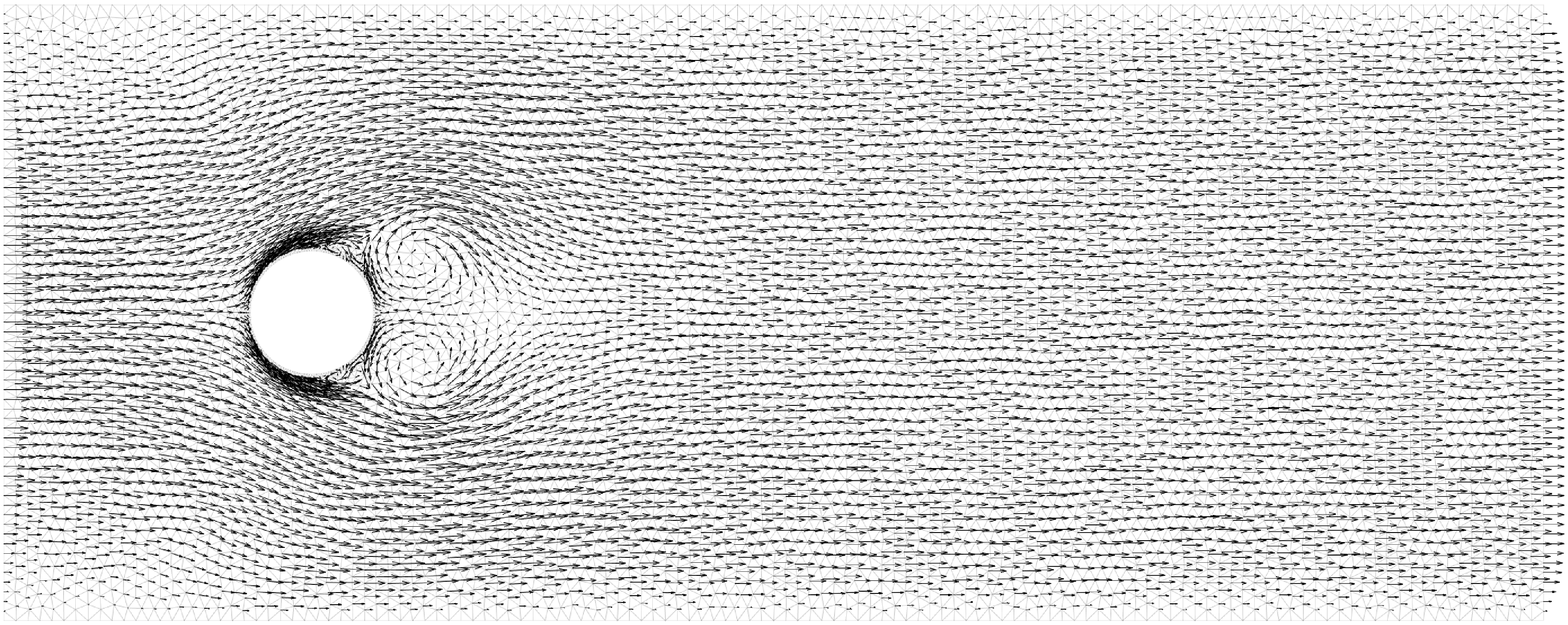}\\
  \end{tabular}
		
	\caption{Velocities after 15, 20, and 30 timesteps from top, $\mu=10^{-6}$, $\gamma = 10^{-1}$.}
	\label{fig:velstab}
\end{figure}
\begin{figure}[ht]
	\begin{center}
		\includegraphics[scale=0.25]{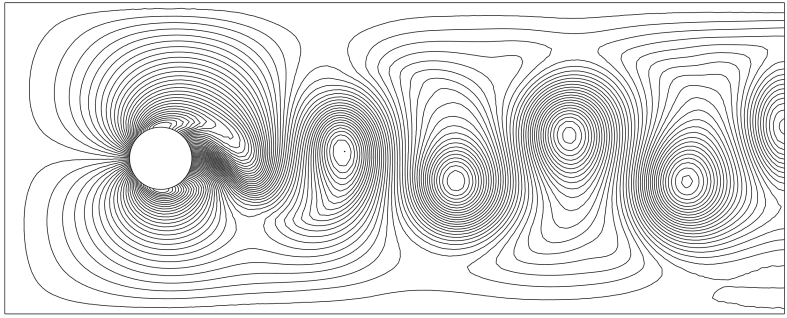}
	\end{center}
	\caption{Relative streamlines at time $t=10$, $\mu =10^{-6}$, $\gamma = 10^{-1}$.}
	\label{fig:vort1e-6}
\end{figure}

\end{document}